\documentclass[final,11pt]{siamart171218}
\usepackage{amsmath}
\usepackage{amssymb}
\usepackage{geometry}
\usepackage{graphicx}
\usepackage{float}
\usepackage{siunitx}
\usepackage{tikz, pgfplots, adjustbox}
\usetikzlibrary{intersections}
\usetikzlibrary{calc}
\usepackage{algorithm}
\usepackage{algpseudocode}
\usepackage{tikz}
\usepackage{cancel}
\usepackage{cleveref}
\usepackage[most]{tcolorbox}
\usepackage[sort,nocompress]{cite}
\newtheorem{assumption}[theorem]{Assumption}

\usepackage{hyperref}
\usepackage{bookmark}

\newcommand{\eps}{\varepsilon}

\newcommand{\bs}[1]{{\boldsymbol #1}}

\colorlet{C1}{green!30!yellow}
\colorlet{C2}{green!90!black}
\colorlet{C3}{green!60!black}

\graphicspath{{./figures/}}

\newcommand{\R}[0]{\mathbb{R}}
\newcommand{\Vr}{\bs V_{\!\!r}}
\newcommand{\Vrp}{\bs V_{\!\!r^+}}
\newcommand{\Vhatr}{\hat{\bs V}\mathstrut_{\!\!r}}
\newcommand{\Vhatrp}{\hat{\bs V}\mathstrut_{\!\!r^+}}
\newcommand{\Wr}{\bs W_{\!\!r}}

\title{Amortized low-rank approximation for hyperparameter marginalization in PDE-governed Bayesian inverse problems\thanks{Version from \today.}
\funding{Partially supported by the Multidisciplinary University Research Initiatives (MURI) Program Office of Naval Research (ONR) grant \#N00014-19-1-242, by ONR \#N00014-26-1-2101, by the US National Science Foundation (NSF) under \#2411229, and by the U.S.\ Department of Energy, Office of Science, Office of ASCR, DOE Computational Science Graduate Fellowship under \#DE-SC0022158.
}}

\author{Sonia Reilly\thanks{Courant Institute School of Mathematics, Computing and Data Science, New York University, New York, USA, \email{sonia.reilly@nyu.edu, stadler@cims.nyu.edu}}
 \and Georg Stadler\footnotemark[2]
}
\headers{Amortized hyperparameter marginalization}{Sonia Reilly, Georg Stadler}

\pgfplotsset{compat=1.18}
\usepackage{url}
\begin{document}
\maketitle
\begin{abstract}
This paper addresses the efficient solution of hierarchical Bayesian inverse problems with a high- or infinite-dimensional parameter field and a moderate number of hyperparameters. We focus on a class of problems in which the parameter-to-observable mapping is a linear PDE, so that, for fixed hyperparameters, the problem becomes conditionally Gaussian. Marginalizing the hyperparameters entails repeated evaluation of their marginal density, which in turn entails repeated large-scale log-determinant ratio computations and maximum a posteriori (MAP) estimations. To address these computational challenges, we introduce a scalable framework that relies on generalized low-rank approximations of the update from the prior to the posterior precision matrix. Though such a framework is state-of-the-art in non-hierarchical settings, in the case of nonlinear prior hyperparameters such as covariance length scales, a direct application to hierarchical problems is inefficient. We propose two amortized variants of this framework, comparing their theoretical properties and computational complexity to the direct method. In numerical experiments, we evaluate their performance on an advection–diffusion initial condition inverse problem in two and three spatial dimensions, with hyperparameters in both the prior and the noise covariance. We find that our amortized methods achieve a factor of 30--45 speedup for 100 marginal density evaluations relative to the direct method in the 3D problem.
\end{abstract}



\section{Introduction} In Bayesian inverse problems where the forward operator is costly to compute, often because it requires solving a PDE, one usually distinguishes between linear Gaussian settings and nonlinear, potentially non-Gaussian ones. In the linear Gaussian case, the posterior distribution is available in closed form and can be accurately approximated even when the parameter dimension is very large or infinite. The nonlinear case typically necessitates the use of MCMC or importance sampling, often combined with sophisticated proposal or dimension reduction techniques.
An important class of problems involving hyperparameters lies between these two extremes. While considering the full set of parameters leads to a nonlinear problem, there is frequently an underlying linear structure: for fixed hyperparameters, the conditional distributions have a linear Gaussian form. If the hyperparameter dimension is moderate, this structure can be used to devise approximation methods that do not require sampling, or only require sampling in the space of hyperparameters. We now formalize the class of problems we consider.

\subsection{Problem statement}\label{sec:problem}
We define a map $A:X\rightarrow \mathbb R^q$ with $X$ a separable Hilbert space and $q\ge 1$. We consider the problem of inferring the parameter $m \in X$ and a vector of hyperparameters $\bs\theta \in \mathbb R^k$ from observations $\bs y\in \mathbb R^q$ using the linear relation
\begin{equation}\label{eq:IP}
    \bs y = Am + \bs \varepsilon,
\end{equation}
where $\bs\varepsilon$ is noise that corrupts the observations, and bold symbols represent vectors and matrices to differentiate them from functions and operators. Adopting a hierarchical Bayesian approach, we consider $m,\bs \theta,\bs \varepsilon$ to be random variables with known prior distributions. We assume Gaussian noise $\bs \varepsilon\sim\mathcal N(\bs 0,\bs Q_{\bs\varepsilon}^{-1}(\bs\theta))$ and a Gaussian prior $m\sim\mathcal N(\mu_{\text{pr}}(\bs \theta),Q_\text{pr}^{-1}(\bs\theta))$ with $\mu_\text{pr}(\bs\theta)\in X$  and trace-class covariance operator $Q_\text{pr}^{-1}(\bs\theta)$ for all $\bs\theta$.

Writing the distributions of $m, \bs \theta$, and $\bs y$ respectively as $\mathbb{P}(dm), \mathbb{P}(\bs \theta)$, and $\mathbb{P}(\bs y)$, we characterize $\bs \theta$ using the hyperparameter marginal given by Bayes' Law,
\begin{equation} \label{eq:hierarchical_Bayes}
    \mathbb{P}(\bs \theta|\bs y) \propto \frac{\mathbb{P}(dm, \bs \theta, \bs y)}{\mathbb{P}(dm|\bs \theta, \bs y)}
    = \frac{\mathbb{P}(\bs y|dm, \bs \theta) \mathbb{P}(dm|\bs \theta) \mathbb{P}(\bs \theta)}{\mathbb{P}(dm|\bs \theta, \bs y)},
\end{equation}
where we write $dm$ rather than $m$ because $X$ can be infinite dimensional and its density may not exist.
To characterize $m$, we compute its marginal distribution
given by the integral
\begin{equation} \label{eq:param_marginal}
    \mathbb{P}(dm|\bs y)
    = \int \mathbb{P}(dm| \bs \theta, \bs y) \mathbb{P}(\bs \theta| \bs y) \, d \bs \theta.
\end{equation}
Because the conditional posterior of $m$ given $\bs\theta$ is Gaussian, this integral can be approximated by a Gaussian mixture using quadrature or sampling over the low-dimensional $\bs\theta$. 
An alternative approach \cite{chung_efficient_2024, hall-hooper_efficient_2024} is the empirical Bayes approximation, which finds the $\bs \theta^\ast$ that maximizes the hyperparameter marginal and approximates the marginal of $m$ by conditioning on $\bs \theta^\ast$. 

Both approaches hinge on being able to efficiently evaluate and explore $\mathbb{P}(\bs\theta|\bs y)$, as well as to repeatedly carry out computations involving the posterior $\mathbb{P}(dm|\bs\theta,\bs y)$. Developing efficient methods to accomplish this for PDE-governed inverse problems is the central objective of this paper.

\subsection{Related literature}
Hierarchical Bayesian inverse problems have garnered considerable attention in statistics and uncertainty quantification in recent years. Drawing samples from the full joint posterior quickly becomes impractical as the dimension increases, so an alternative is to integrate out the high-dimensional latent field. This usually relies on the same assumption as in this work: conditional on the hyperparameters, the latent variable is (approximately) Gaussian.
For example, the MCMC approaches in \cite{saibaba_efficient_2019, brown_low-rank_2018, buser_efficient_2025} construct Gibbs-type samplers that exploit marginalization to make sampling computationally practical. Specifically, in \cite{saibaba_efficient_2019, brown_low-rank_2018} the authors employ the same low-rank approximation of the precision update as we do; however, they make the assumption that the hyperparameters appear only linearly in the prior, which enables precomputation of the low-rank approximation of the prior-preconditioned precision update. In \cite{buser_efficient_2025}, the authors assume that the precision update is not sufficiently low-rank to be accurately approximated. Instead, they employ a Golub–Kahan bidiagonalization scheme to derive an alternative approximation of the marginal and also introduce a method for computing gradients of this marginal.

Other studies, including \cite{chung_efficient_2024, hall-hooper_efficient_2024}, forgo sampling over the hyperparameter space altogether and instead use an empirical Bayes strategy: they maximize the hyperparameter marginal and then condition on the maximizing hyperparameter. The authors of \cite{hall-hooper_efficient_2024} again employ a Golub-Kahan bidiagonalization, but they recompute the approximation for every new choice of $\bs \theta$. By contrast, \cite{chung_efficient_2024} focuses on the efficient evaluation of the hyperparameter marginal, relying on Monte Carlo estimation to approximate the log-determinant term in the marginal density. Both works also develop methods for estimating the gradient of the hyperparameter marginal, enabling gradient-based empirical Bayes optimization. Related analytic work investigates the consistency of MAP estimation for hierarchical problems in infinite-dimensional settings  \cite{DunlopHelinStuart20}.

A broader class of hierarchical Bayesian inverse problems is addressed in \cite{NortonChristenFox18, fox_fast_2016}, where the authors introduce marginal-then-conditional sampling as an alternative to standard Gibbs sampling. Their framework does not impose a linear Gaussian structure on the inverse problem for fixed hyperparameters, allowing for any conjugate priors and likelihoods. However, it does make the assumption that the hyperparameter marginal is fast to sample, and the focus is on problems with simple forward operators. This type of sampling could perhaps be combined with the methods we present here to extend it to PDE-governed problems.

Efficient marginalization over hyperparameters, and in particular the evaluation of determinant ratios, is also an important task in Gaussian process regression \cite{DongEriksson,RasmussenWilliams}. However, GP regression operates in a rather different cost regime, where the forward operators allow for fast evaluation.
Similarly, Integrated Nested Laplace Approximation (INLA) \cite{rue2009approximate,RueRieblerSorbyeEtAl17, Opitz17}, a widely used method for hierarchical problems in spatial statistics, addresses problems similar to ours but relies on sparsity that PDE forward operators lack. As INLA is less well known in the PDE inverse problem community, we present a more detailed discussion and comparison with INLA in \cref{sec:relation_INLA}. 

Finally, our method builds heavily on the literature developing
low-rank approximations of the precision update as an approach to
efficient solution of Bayesian inverse problems, in particular
\cite{Bui-ThanhGhattasMartinEtAl13, spantini_optimal_2015,
  CarereLie25, CuiMartinMarzouk14, AlexanderianSaibaba}. Our approach
extends prior-preconditioned precision updates
\cite{Bui-ThanhGhattasMartinEtAl13, spantini_optimal_2015,
  CarereLie25} which are derived through solving generalized
eigenvalue problems.
Amortization strategies similar to what we propose also appear in the
literature on optimal experimental design (OED) for PDE-governed
Bayesian inverse problems, where the design criterion is typically the
trace or log-determinant of the posterior covariance matrix, which
must be evaluated repeatedly as the design varies
\cite{AlexanderianSaibaba, AlePetStaGha14}.
A main difference is that in OED, the prior and forward operators remain fixed, and only the design, implemented through observation weighting, changes.

\subsection{Contributions and limitations}
We propose a framework for fast marginalization over the hyperparameters $\bs \theta$ based on low–rank approximation techniques for Bayesian inverse problems \cite{Bui-ThanhGhattasMartinEtAl13, spantini_optimal_2015, CarereLie25}. Such techniques typically depend on approximating a prior-preconditioned precision update, which complicates their application to evaluating $\pi(\bs \theta|\bs y)$ when the prior depends on $\bs \theta$ (see \cref{sec:pi_theta_and_PP}). We show in \cref{sec:WP_and_UP} that the prior can be substituted with a $\bs \theta$-independent preconditioner to enable precomputation with the parameter-to-observable map, thereby amortizing the approximation cost over many evaluations of $\pi(\bs \theta|\bs y)$. This makes it well-suited for problems with hyperparameters that enter nonlinearly in the prior, such as correlation lengths. 
We establish theoretical guarantees by upper-bounding the truncation error of the proposed methods, thus indicating a principled choice of truncation rank, and set out the conditions that alternative preconditioners must satisfy. We demonstrate the efficiency of this approach compared to direct approximation of the prior-preconditioned update.
Our framework to approximate posteriors in Bayesian problems with hyperparameters differs from existing methods in that it does not rely on MCMC or other sampling.
Because it is fully deterministic, it also enables a quadrature-based improvement over empirical Bayes estimation. It is designed for computationally demanding PDE-governed problems, and in \cref{sec:adv_diff} we illustrate its effectiveness on a three-dimensional, time-dependent inverse problem, for which we achieve a 30--45$\times$ speedup relative to the prior-preconditioned method.

In developing our framework, we make assumptions that impose some limitations. In particular, the current framework targets linear forward operators, so that the conditional posterior is Gaussian. To avoid recomputing the precision update for each hyperparameter marginal evaluation or relying on nonlinear operator surrogates, we assume that both the forward operator and the noise precision $\bs Q_{\bs \varepsilon}$ depend on $\bs \theta$ multiplicatively or not at all. Finally, although our methods can be applied to problems where $\bs Q_\text{pr}$ depends on $\bs \theta$ solely in a multiplicative way, this is not the target setting. Our proposed algorithms do not provide computational gains in this simpler case, since the prior-preconditioned update can be approximated in advance. 


\section{Preliminaries} \label{sec:prelim}
We begin by discussing PDE-governed infinite-dimensional Bayesian inverse problems and their discretization, as well as introducing relevant computational tools. We underscore key assumptions about the target problem and compare it to problems that can be addressed with other methods.

\subsection{Discretization of infinite-dimensional problems}
The problem described in \cref{sec:problem} is a PDE-governed inverse problem if the application of $A$ requires a PDE solution. In this context, the parameter $m$ typically represents a function, such as a spatially-varying coefficient, initial condition, or boundary condition. Discretizing this function yields a parameter vector $\bs m$ whose dimension can be very large, posing a significant computational challenge.

For concreteness, let $\bs m$ be the vector of finite element coefficients of a function $m(x)$. To ensure discretization-invariance, we let $\bs m \in \R^n_{\bs M}$, the space $\R^n$ endowed with the inner product $\langle\cdot\,, \cdot \rangle_{\bs M}$ weighted by the finite element mass matrix $\bs M$ to approximate the $L^2$ inner product. To distinguish adjoints in this weighted space from transposes in standard $\R^n$, we follow the notation of \cite{stuart_inverse_2010} and denote the adjoints of $\bs F: \R^n_{\bs M} \to \R^q$ and $\bs V: \R^q \to \R^n_{\bs M}$ as
\begin{equation}\label{eq:adjoint}
\begin{aligned} 
    \bs F^\sharp &:= \bs M^{-1} \bs F^\top, \\
    \bs V^\diamond &:= \bs V ^\top \bs M.
\end{aligned}
\end{equation}
Note also that in $\R^n_{\bs M}$, precision operators are matrices that are self-adjoint and positive definite with respect to the $\bs M$-weighted inner product. In particular, the discretization $\bs Q$ of a precision operator $Q$ can be written as $\bs Q = \bs M^{-1} \bs B$ for some symmetric positive definite $\bs B \in \R^{n \times n}$. Then, self-adjointness of $\bs Q$ follows from
\begin{equation*}
    \langle \bs x, \bs Q \bs y \rangle_{\bs M} 
    = \bs x^\top \left( \bs M \bs M^{-1} \bs B \bs y \right) 
    = \bs x^\top \bs B \bs y 
    = \big(\bs M^{-1} \bs B \bs x \big)^\top \bs M \bs y 
    = \langle \bs Q \bs x,  \bs y \rangle_{\bs M}.
\end{equation*}
We adopt a discretization-invariant formulation in which operators are self-adjoint in the $\bs M$-weighted inner product, and the prior covariances are trace-class. The importance of posing hierarchical inverse problems so that hyperparameters remain interpretable in the continuum limit, and of understanding algorithmic behavior under mesh refinement, was established by \cite{Agapiou2014}, who showed that standard
Gibbs samplers for such problems can degrade as the discretization is refined.

\subsection{Linear Gaussian Bayesian inverse problems}
\label{sec:lineargaussiancase}
Upon discretization, \eqref{eq:IP} becomes
\begin{equation*}
    \bs y = \bs A \bs m + \bs \varepsilon,
\end{equation*}
where $\bs A \in \R^{m \times n}$ is the forward operator and $\bs m \in \R^n$ are inversion parameters. For fixed hyperparameter $\bs \theta$, we obtain a Gaussian likelihood $\pi(\bs y | \bs m, \bs \theta)$ given by $\mathcal{N}(\bs A \bs m, \bs Q_{\bs \varepsilon}^{-1}(\bs \theta))$, a Gaussian prior distribution $\pi(\bs m|\bs \theta)$ given by $\mathcal{N}(\bs \mu_\text{pr}(\bs \theta), \bs Q_\text{pr}^{-1}(\bs \theta))$, and a linear forward map. Together, these make $\pi(\bs m | \bs \theta, \bs y)$ a linear Gaussian Bayesian inverse problem. 
By combining these known distributions according to Bayes' law, the posterior distribution is, up to a multiplicative constant, given by
\begin{equation*}
    \pi(\bs m|\bs \theta, \bs y) \propto \pi(\bs y | \bs m, \bs \theta) \pi(\bs m|\bs \theta).
\end{equation*}
In the linear Gaussian case, the posterior is a Gaussian distribution $\mathcal{N}(\bs \mu_\text{post}(\bs \theta), \bs Q_\text{post}^{-1}(\bs \theta))$ with well-known expressions for the precision and mean \cite{Bui-ThanhGhattasMartinEtAl13},
\begin{align}
\label{eq:post-precision}
    \bs Q_\text{post} &= \bs Q_\text{pr} + \bs A^\sharp \bs Q_{\bs \varepsilon} \bs A   \\
    \bs \mu_\text{post} &= \bs Q_\text{post}^{-1} (\bs Q_\text{pr} \bs \mu_\text{pr} + \bs A^\sharp \bs Q_{\bs\varepsilon} \bs y),
    \label{eq:post-mean}
\end{align}
where we have dropped the $\bs \theta$ dependence for clarity, as we will from here on. 
The term $\bs A^\sharp \bs Q_{\bs \varepsilon} \bs A$ will be of special interest since it represents the information gain between the prior and posterior due to the data. This term is sometimes called the ``misfit Hessian''. It is often concentrated onto a low dimensional subspace and therefore low rank \cite{Bui-ThanhGhattasMartinEtAl13}; we refer to as the \emph{precision update}.

\subsection{Dominating operations in Bayesian inference governed by PDEs}
\label{sec:complexity}
Next, we review a few of the computational considerations of typical linear PDE-governed Bayesian inverse problems, adopting a 
framework like that of \cite{Bui-ThanhGhattasMartinEtAl13}. In particular, we make the following assumption for the problem with fixed hyperparameters $\bs\theta$.
\begin{assumption} \label{ass:complexity}
We do not have access to $\bs A, \bs A^\sharp, \bs Q_\text{pr}$, or $\bs Q_\text{pr}^{-1}$ as matrices, but we can apply them to vectors.
\begin{itemize}
    \item The complexity of applications of $\bs A$ and $\bs A^\sharp$ grows as $O(n^{\alpha})$ for $\alpha\ge 1$; if the scaling is linear ($\alpha=1$), the leading constant is large.
    \item The complexity of applications of $\bs Q_\text{pr}$ and $\bs Q_\text{pr}^{-1}$ grows as $O(n)$.  If $\alpha=1$, the leading constant of prior covariance operations is substantially smaller than of application of $A$.
\end{itemize}
\end{assumption}

The parameter-to-observable map $\bs A$ represents a linear or linearized steady-state or time-dependent PDE solve. For example, in the numerical experiments discussed in \cref{sec:adv_diff}, the application of $\bs A$ (and also its adjoint) requires solving a time-dependent advection-diffusion equation using implicit time stepping. As another example, in \cite{Bui-ThanhGhattasMartinEtAl13}, the application of $\bs A$ (and its adjoint) amounted to time-stepping a linear coupled acoustic-elastic wave equation.

Typical choices for Gaussian priors in infinite-dimensional Bayesian inverse problems are based on
Mat{\'e}rn kernels due to their relation to stochastic elliptic PDEs.
In particular, in \cite{lindgren_explicit_2011} the authors show that certain Mat{\'e}rn kernels correspond to precision operators that take the form of Laplacian-like PDEs with integer powers. Upon discretization, these PDEs result in sparse matrix systems. Due to the sparsity, the application of these matrices scales linearly with $n$. Moreover, application of the covariance amounts to solving systems with elliptic PDEs, for which fast PDE solvers (such as multigrid) are available. These solves typically have an $\mathcal O(n)$ time complexity, with low constants.

\subsection{Generalized eigenvalue problems and low-rank approximation}
Since the matrix $\bs A$ is only available through matrix-vector multiplications, we approximate the posterior precision via low rank approximations of the prior precision update, computed using only matrix-vector products \cite{Bui-ThanhGhattasMartinEtAl13}; see \cref{sec:prior-precon-update} for more on this. In particular, \cite{Bui-ThanhGhattasMartinEtAl13, spantini_optimal_2015} suggest approximating a prior preconditioned form of the precision update, denoted as 
\begin{equation*}
    \bs Q_\text{pr}^{-1/2} \bs A^\sharp \bs Q_{\bs \varepsilon} \bs A \bs Q_\text{pr}^{-1/2}.
\end{equation*}
Given the ability to compute the square root of the prior covariance, this could be achieved using a randomized SVD method \cite{HalkoMartinssonTropp11}. Instead, to avoid evaluation of the square root prior covariance, we follow \cite{villa_hippylib_2018} and use the generalized eigenvalue problem method described in \cite{saibaba2016randomized}.

A symmetric generalized eigenvalue problem is a problem of the form
\begin{equation*}
    \bs H \bs X = \bs \Lambda \bs B \bs X,
\end{equation*}
where $\bs H$ is symmetric and $\bs B$ is symmetric positive definite.
Letting $\bs B = \bs L\bs L^\top$ be the Cholesky decomposition of $\bs B$ and $\bs Y = \bs L^\top \bs X$, the problem can be rewritten as
\begin{equation*}
    \bs L^{-1} \bs H \bs L^{-\top} \bs Y = \bs \Lambda \bs Y
\end{equation*}
by multiplying both sides by $\bs L^{-1}$. As the eigenvectors of a symmetric matrix, the columns of $\bs Y$ are orthonormal; it follows that the generalized eigenvectors $\bs X$ are orthonormal in the $\bs B$-weighted inner product, i.e., $\bs X^\top \bs B \bs X = \bs I$.
The authors of \cite{saibaba2016randomized} propose a randomized single-pass algorithm for such generalized symmetric eigenvalue problems. To compute the largest $r$ generalized eigenvalues and corresponding eigenvectors, they compute the product of $\bs B^{-1} \bs H$ with $r+p$ random vectors, where $p$ is an oversampling factor included for accuracy. The remaining core operations of the algorithm are a QR factorization with $\bar{\bs Q}$ orthogonal in the $\bs B$-inner product and a standard eigenvalue problem---see \cite[Alg.~7]{saibaba2016randomized}. This is the method we use to compute the low rank decomposition of a preconditioned precision update.

\subsection{Relation to INLA}
\label{sec:relation_INLA}
Finally, we highlight the connections between our approach and that of the integrated nested Laplace approximation (INLA). INLA is a widely used method for hierarchical Bayesian inference in the statistics community \cite{rue2009approximate, rueheld2005}. Similar to the setting above, INLA is designed for latent Gaussian models and employs Mat\'ern-type priors with sparse precision matrices. It can be applied to a somewhat broader range of problems than those considered here, since it uses Laplace approximations for fixed hyperparameters. However, it relies crucially on the sparsity of both the prior and posterior precision matrices, which limits its applicability to inverse problems stemming from PDEs. More precisely, INLA
represents the posterior by nesting two layers: an inner Laplace approximation of the latent field conditional on the hyperparameters, combined with an outer, low-dimensional numerical quadrature over the (typically few) hyperparameters. The inner Laplace step in INLA coincides with the Gaussian approximation at the MAP point for fixed hyperparameters; in the notation of \eqref{eq:hierarchical_Bayes}, this means applying a Laplace approximation to $\mathbb{P}(dm|\bs\theta,\bs y)$. Since we assume a linear forward map and restrict hyperparameters to the prior and noise covariance, the resulting posterior is Gaussian, and thus, this approximation is exact. 

INLA cannot directly address inference problems governed by PDEs due to a structural limitation: it assumes that the high-dimensional parameters enter the likelihood via \emph{local}, linear predictors with conditionally independent observations. Conversely, PDE parameter-to-observable maps are generally \emph{non-local}: each observation depends on the global PDE solution and, consequently, the full parameter field. INLA takes advantage of the fact that the posterior precision matrix \eqref{eq:post-precision} inherits the sparsity pattern of the prior precision. In particular, $\bs A^\sharp \bs Q_{\bs \varepsilon} \bs A$ contributes only a diagonal term, which is not the case when $\bs A$ arises from a PDE solution operator. INLA leverages this sparsity together with efficient numerical linear algebra techniques—most notably sparse Cholesky factorization with a fixed sparsity pattern—to compute the determinant of the posterior precision \cite{rue2009approximate, Gaedke23}. In contrast, Bayesian inverse problems constrained by PDEs generally lack such sparsity; in this setting, low-rank approximations as exploited in this work have proven to be a highly effective alternative tool.

\section{Approximation of the hyperparameter marginal} \label{sec:pi_theta_and_PP}

Inference over both the hyperparameters and the parameter $\bs m$ requires repeated evaluation of the hyperparameter marginal, which can be expressed in terms of known densities by \eqref{eq:hierarchical_Bayes}.
A straightforward computation yields
\begin{equation}\label{eq:pi(theta|y)}
    \pi(\bs\theta |\bs y) \propto \left( \frac{|\bs Q_\text{pr}||\bs Q_{\bs\varepsilon}|}{|\bs Q_\text{post}|} \right)^{1/2} \exp \left( -\frac{1}{2} \left[ \|\bs y\|^2_{\bs Q_{\bs\varepsilon}} + \|\bs \mu_\text{pr}\|^2_{\bs M \bs Q_\text{pr}} - \|\bs \mu_\text{post}\|^2_{\bs M \bs Q_\text{post}} \right] \right) \pi(\bs\theta),
\end{equation}
where the $\bs M$ in the norms accounts for the $\bs M$-inner product of the space $\R^n_{\bs M}$.
We refer to \cite{DunlopHelinStuart20} for an equivalent expression for infinite-dimensional parameters.

Each evaluation of $\pi(\bs \theta|\bs y)$ requires two expensive steps: computation of the posterior mean $\bs \mu_{\text{post}}$ and computation of a determinant ratio. The former requires solving a system with the posterior precision $\bs Q_\text{post}$, see \eqref{eq:post-precision} and \eqref{eq:post-mean}, where each application of $\bs Q_\text{post}$ involves two PDE solves. The latter may be costly and unstable since the determinants of the prior and posterior precision, as inverses of discretized trace-class operators, diverge for $n\to \infty$.
Since by \cref{ass:complexity} the dominant cost in computing $\pi(\bs \theta | \bs y)$ is the application of the parameter-to-observable map $\bs A$, we seek methods that amortize the cost of these applications across many evaluations of $\pi(\bs \theta | \bs y)$ by (1) precomputing approximations of terms containing $\bs A$, independent of $\bs \theta$, and (2) evaluating $\pi(\bs \theta | \bs y)$ using these approximations, minimizing applications of $\bs A$ as much as possible. These methods should also allow for stable computation of the determinant ratio for any parameter dimension $n$.

The key idea is to precompute low rank approximations of variants of the prior-to-posterior precision update $\bs Q_\text{post} - \bs Q_\text{pr} = \bs A^\sharp \bs Q_{\bs \varepsilon} \bs A$. For many inverse problems, a low- or moderate-dimensional approximation to this term is possible due to the finite number of observations and other properties of the discretized parameter-to-observable map $\bs A$ that are typical in inverse problems. Such low-rank ideas have enabled the solution of high-dimensional linearized Bayesian inverse problems \cite{Bui-ThanhGhattasMartinEtAl13,spantini_optimal_2015}, and have also become a useful tool for nonlinear problems by identifying moderate-dimension likelihood-informed subspaces \cite{CuiMartinMarzouk14}. 

Below we analyze the truncation error introduced by such an approximation and state required assumptions that enable this precomputation. Then we discuss
a previously established approximation method that has a significant drawback for our problem.

\subsection{Approximation error}\label{sec:approx-error}
Given an approximation $\bs Q_{\text{post},r}$ 
for the posterior precision and the corresponding approximation $\bs \mu_{\text{post},r}$ for the posterior mean,
the error introduced by substituting $\bs Q_{\text{post},r}$ and $\bs \mu_{\text{post},r}$ into $- \log \pi(\bs \theta|\bs y)$ is $e(r,\bs \theta) := e_1(r,\bs \theta) + e_2(r, \bs \theta)$, where
\begin{align}
    e_1(r,\bs \theta) &:= - \frac{1}{2} \log \frac{|\bs Q_\text{pr}|}{|\bs Q_\text{post}|} + \frac{1}{2} \log \frac{|\bs Q_\text{pr}|}{|\bs Q_{\text{post},r}|} 
    = \frac{1}{2} \log \frac{|\bs Q_\text{post}|}{|\bs Q_{\text{post},r}|} \label{eq:e1}\\
    e_2 (r,\bs \theta) &:= -\frac{1}{2} \| \bs \mu_\text{post}\|^2_{\bs M \bs Q_\text{post}} +  \frac{1}{2} \| \bs \mu_{\text{post},r}\|^2_{\bs M \bs Q_{\text{post},r}}
    =  \frac{1}{2} \|\bs Q_\text{pr} \bs \mu_\text{pr} + 
    \bs A^\sharp \bs Q_{\bs \varepsilon} \bs y\|^2_{\bs M (\bs Q^{-1}_{\text{post},r} - \bs Q^{-1}_\text{post})}. \label{eq:e2}
\end{align}
For each method we describe, we discuss how to bound $e_1$ by choosing $r$ sufficiently large. 
The error $e_2$ depends, among other terms, on the posterior covariance-weighted norm of $\bs Q_\text{pr} \bs \mu_\text{pr}$, which may be significant due to 
the prior precision. However, $e_2$ can still be bounded by choosing $r$ sufficiently large, or by 
computing $\|\bs \mu_\text{post}\|^2_{\bs M \bs Q_\text{post}}$ exactly. This can be done by solving \eqref{eq:post-mean} using an iterative solver such as the conjugate gradient (CG) method with $\bs Q_{\text{post},r}$ as a preconditioner, as is common for linear and linearized Bayesian inverse problems \cite{villa_hippylib_2018,Bui-ThanhGhattasMartinEtAl13}.
Using CG causes $e_2(r,\bs \theta)$ to vanish. In \cref{sec:adv_diff}, we will numerically study the errors $e_1$ and $e_2$ as a function of $r$ for the case where $\bs\mu_\text{pr}=0$. In this case, we find that computing $\bs\mu_\text{post}$ using the low-rank posterior approximation is sufficient, and that the error $e_1$ generally dominates $e_2$ for the same rank $r$ -- see \cref{fig:spectra_error} (right).

\subsection{Assumptions and notation}
We make the following assumptions about the problem structure to analyze different methods for evaluating the hyperparameter marginal. 
\begin{assumption} \label{ass:theta_dependence} The prior-to-posterior update is $\bs \theta$-independent or depends on $\bs \theta$ only in a scalar multiplicative factor. Specifically,
    \begin{itemize}
        \item The noise precision can be written in the form
            \begin{equation*}
            \bs Q_{\bs \varepsilon} = \frac1{\sigma^{2}} \bs P,
            \end{equation*}
        where $\sigma$ may depend on $\bs \theta$ but $\bs P\in \mathbb R^{q\times  q}$ does not. 
        \item $\bs A$ is independent of $\bs \theta$, i.e., the hyperparameters do not enter into the parameter-to-observable map.
    \end{itemize}
\end{assumption}
Note that a scalar factor could be introduced into $\bs A$ as in the noise precision, but since scaling of the forward model is less natural as a hyperparameter, we assume no dependence of $\bs A$ on $\bs \theta$. 

\cref{ass:theta_dependence} limits the possible hyperparameters for which our methods may be applied, in particular, precluding the use of parameters of the PDE as hyperparameters.
However, it the more restrictive cases used in previous work \cite{saibaba_efficient_2019, chung_efficient_2024, buser_efficient_2025, fox_fast_2016}, such as multiplicative hyperparameters confined to the prior or a spatially uniform noise precision.

Importantly, we do not assume that the prior precision has solely multiplicative dependence on the hyperparameters, and in fact we are most interested in cases with more complex dependence on $\bs \theta$.
To underscore this, we separate the multiplicative and non-multiplicative dependence by letting%
\begin{equation}\label{eq:pr_def}
    \bs Q_{\text{pr}} = \delta^2 \bs R^\sharp \bs R,
\end{equation}
where $\delta$ and $\bs R$ may both depend on $\bs \theta$. This decomposition is not unique, but certain choices are more natural -- see \cref{sec:unprecond}. Note that the matrix $\bs R$, while useful from a mathematical standpoint, never needs to be computed explicitly.

With this notation, we can write the posterior precision matrix as
\begin{equation}\label{eq:Qpost-alt}
    \bs Q_\text{post} = \delta^2\bs R^\sharp \bs R +\frac{1}{\sigma^2}\bs A^\sharp \bs P \bs A.
\end{equation}
The second term in \eqref{eq:Qpost-alt} depends on $\bs \theta$ only through the scaling $1/\sigma^2$, while the first term may depend on $\bs\theta$ in a more complicated way. 

\subsection{Prior-preconditioned precision update (PP)}\label{sec:prior-precon-update}
An established method to approximate the posterior covariance $\bs Q_{\text{post}}^{-1}$ starts by rewriting \eqref{eq:Qpost-alt} as
\begin{equation}
    \bs Q_\text{post} = \delta^2 \bs R^\sharp \left(\bs I + \frac{1}{\sigma^2\delta^2} (\bs R^\sharp)^{-1} \bs A^\sharp \bs P \bs A \bs R^{-1} \right) \bs R.
\end{equation}
This suggests a low-rank approximation 
of the \emph{prior-preconditioned update}, $(\bs R^\sharp)^{-1} \bs A^\sharp \bs P \bs A \bs R^{-1}$, whose eigenvalue decomposition can be split as follows:
\begin{equation} \label{eq:pr_precon_approx}
    (\bs R^\sharp)^{-1} \bs A^\sharp \bs P \bs A \bs R^{-1}
    = \bs V \bs \Lambda \bs V^\top = \Vr \bs \Lambda_r \Vr^\top + \Vrp \bs \Lambda_{r^+} \Vrp^\top,
\end{equation}
where $\bs \Lambda := \text{diag}(\lambda_1, \ldots, \lambda_n)$ contains the eigenvalues in descending order, $\bs V \in \R^{n \times n}$ are the eigenvectors, $\bs \Lambda_r := \text{diag}(\lambda_1, \ldots, \lambda_r)$ and $\Vr \in \R^{n \times r}$ form a rank-$r$ approximation, and $\bs \Lambda_{r^+} := \text{diag}(\lambda_{r+1}, \ldots, \lambda_n)$ and $\Vrp \in \R^{n \times n-r}$ form the remainder term.
Then the rank-$r$ updated approximate posterior precision is given by 
\begin{equation}\label{eq:pr_precon_approx_prec}
    \bs Q_{\text{post},r} := \delta^2 \bs R^\sharp \left(\bs I + \frac{1}{\sigma^2\delta^2} \Vr \bs \Lambda_r \Vr^\top \right) \bs R.
\end{equation}
Compared to the (unpreconditioned) precision update $\bs A^\sharp \bs P \bs A$ in \eqref{eq:Qpost-alt}, the \emph{prior-preconditioned} update involves the prior covariance. As a (discretized) trace-class operator, its eigenvalues must decay, so the PP update typically admits a more accurate low-rank approximation at a given rank. In fact, this method has been shown to be optimal in the sense that it minimizes both the Hellinger distance and K-L divergence between the true and the approximate posterior, among approximations with the same rank \cite{spantini_optimal_2015,CarereLie25}.

This posterior precision approximation allows for efficient computation of the posterior covariance and of the determinant ratio term in \eqref{eq:pi(theta|y)}. For the covariance,
\eqref{eq:pr_precon_approx_prec} can be inverted using the Sherman-Morrison-Woodbury identity, resulting in
\begin{equation} \label{eq:pr_precon_approx_cov}
    \bs Q_{\text{post},r}^{-1} = \delta^{-2} \bs R^{-1} \left(\bs I - \Vr \bs D_r \Vr^\top \right) (\bs R^\sharp)^{-1},
\end{equation}
where $\bs D$ is the diagonal matrix with entries $D_{ii} = \frac{\lambda_i}{\lambda_i + \sigma^2 \delta^2}$, and $\bs D_r \in \R^{r \times r}$ and $\bs D_{r^+} \in \R^{n-r \times n-r}$ are the upper left and lower right blocks of $\bs D$, respectively. 
The negative log of the determinant ratio, excluding the easy-to-compute $|\bs Q_{\bs \varepsilon}|$, can be rewritten as 
\begin{equation*} 
    - \log \frac{|\bs Q_\text{pr}|}{|\bs Q_\text{post}|} 
    = \log \det (\bs Q_\text{pr}^{-1} \bs Q_\text{post})
    = \log \det (\delta^{-2}\bs (\bs R^\sharp)^{-1} \bs Q_\text{post} \bs R^{-1})
    = \log \det \left(\bs I + \frac{1}{\sigma^2\delta^2} \bs V \bs \Lambda \bs V^\top \right).
\end{equation*}
Since $\det(\bs I+\bs A\bs B) = \det(\bs I+\bs B\bs A)$ for any $\bs A,\bs B$, and the columns of $\bs V$ are orthonormal,
\begin{equation} \label{eq:pr_precon_det_ratio}
    - \log \frac{|\bs Q_\text{pr}|}{|\bs Q_\text{post}|} 
    = \log \det \left(\bs I + \frac{1}{\sigma^2\delta^2} \bs \Lambda \right) 
    = \sum_{i=1}^n \log \left(1+\frac{\lambda_i}{\sigma^2 \delta^2} \right).
\end{equation}
Following the same argument, we can approximate this term by
\begin{equation} \label{eq:pr_precon_det_ratio_approx}
    - \log \frac{|\bs Q_\text{pr}|}{|\bs Q_{\text{post},r}|}
    = \sum_{i=1}^r \log \left(1+\frac{\lambda_i}{\sigma^2 \delta^2} \right).
\end{equation}
Note that prior-preconditioning avoids the $n \to \infty$ divergence noted above, since the eigenvalues of the preconditioned update are summable as a consequence of the trace-class prior covariance.

The resulting error in the log-determinant component can therefore be bounded as:
\begin{equation} \label{eq:pr_precon_det_ratio_error}
    e_1(r,\bs \theta) =  \frac{1}{2} \sum_{i=r+1}^n \log \left(1+\frac{\lambda_i}{\sigma^2 \delta^2} \right) \leq \frac{1}{2 \sigma^2 \delta^2} \sum_{i=r+1}^n \lambda_i.
\end{equation}
Thus, to obtain a small error $e(r, \bs \theta)$, 
one can choose $r$ such that $\lambda_{r+1} \ll \sigma^2 \delta^2$.

To control the error term $e_2$, one can either compute $\bs \mu_\text{post}$ exactly using CG, preconditioned with $\bs Q_{\text{post},r}$ (as discussed above), or directly use the low-rank posterior approximation to approximate $\bs \mu_\text{post}$. Following the latter route and
defining $\bs g := \bs Q_\text{pr} \bs \mu_\text{pr} + 
    \bs A^\sharp \bs Q_{\bs \varepsilon} \bs y$, that error is
\begin{equation} \label{eq:pr_precon_CG_error}
\begin{aligned}
    e_2 (r,\bs \theta)
    &= \frac{1}{2} \|\bs g\|^2_{\bs M (\bs Q^{-1}_{\text{post},r} - \bs Q^{-1}_\text{post})}
    =  \frac{1}{2 \delta^2} \bs g^\top \bs M \bs R^{-1} \Vrp \bs D_{r^+} \Vrp^\top \bs (\bs R^\sharp)^{-1} \bs g \\
    &\leq \frac{\lambda_{r+1}}{2(\lambda_{r+1} + \sigma^2 \delta^2)} \|\bs g\|^2_{\bs M \bs Q^{-1}_\text{pr}},
\end{aligned}
\end{equation}
because 
\begin{equation*}
    \max_{\bs x \in \R^{n-r}, \bs x \neq 0} \frac{\bs x^\top \Vrp \bs D_{r^+} \Vrp^\top \bs x}{\bs x^\top \bs x} = \frac{\lambda_{r+1}}{\lambda_{r+1} + \sigma^2 \delta^2}.
\end{equation*}
Thus, for $e_2(r,\bs \theta)$ to be small (and no iterative solver to be necessary),
$r$ should be chosen large enough that $\lambda_{r+1} \ll \sigma^2 \delta^2 (\|\bs g\|^2_{\bs M \bs Q^{-1}_\text{pr}}-1)^{-1}$. Note that the right hand side in this estimate depends on the data and the prior mean.


For problems with hyperparameters that do not appear in the prior precision or that enter only in the scaling factor $\delta^2$, PP is the optimal method for low-rank approximation. However, it is not a good choice for problems where $\bs R$ depends on $\bs \theta$, since the prior-preconditioning adds a dependence on $\bs \theta$ to the update. Therefore, for each $\bs \theta$, a new low-rank approximation would have to be computed. The next section focuses on efficient methods to avoid such recomputation.

\section{Amortized approximation of $\boldmath{\pi(\bs \theta|\bs y)}$} \label{sec:WP_and_UP}
To amortize computational costs across many evaluations of $\pi(\bs \theta|\bs y)$ when $\bs R$ depends on $\bs \theta$,
we introduce and analyze a framework that exchanges the optimality of the PP approximation for a higher rank $\bs \theta$-independent approximation that allows for precomputation. In \cref{sec:weakest,sec:unprecond} we describe two variants of this framework and in \cref{sec:complexity-comparison} we compare them to PP with respect to their computational cost.

\subsection{Weakest prior-preconditioned precision update (WP)}\label{sec:weakest}
An alternative to computing the prior-preconditioned update for each $\bs \theta$ is to use an upfront approximation of the update preconditioned with a $\bs \theta$-independent prior.
Namely, we use a
reference prior precision $\hat{\bs Q}_\text{pr} := \hat{\bs R}\mathstrut^\sharp \hat{\bs R}$, in which we choose $\hat{\delta} = 1$ without loss of generality.
We denote the eigendecomposition of the update preconditioned by this reference prior by
\begin{equation} \label{eq:wk_precon_decomp}
    (\hat{\bs R}\mathstrut^\sharp)^{-1} \bs A^\sharp \bs P \bs A \hat{\bs R}\mathstrut^{-1}
    = \hat{\bs V} \hat{\bs \Lambda} \hat{\bs V}\mathstrut^\top 
    = \Vhatr \hat{\bs \Lambda}_r \Vhatr^\top + \Vhatrp \hat{\bs \Lambda}_{r^+} \Vhatrp^\top.
\end{equation}
After the upfront computation of the rank-$r$ approximation, we approximate the posterior precision for each $\bs \theta$ by
\begin{equation} \label{eq:wk_precon_approx_prec}
    \hat{\bs Q}_{\text{post},r} 
    := \delta^2 \bs R^\sharp \left(\bs I + \frac{1}{\sigma^2\delta^2} (\bs R^\sharp)^{-1} \hat{\bs R}\mathstrut^\sharp \Vhatr \hat{\bs \Lambda}_r \Vhatr^\top \hat{\bs R} \bs R^{-1} \right) \bs R,
\end{equation}
i.e., we replace the $\hat{\bs Q}_\text{pr}$ reference preconditioning with the target $\bs Q_\text{pr}$ preconditioning. This approximation is nearly of the form \eqref{eq:pr_precon_approx_prec}, except that the columns of $(\bs R^\sharp)^{-1} \hat{\bs R}\mathstrut^\sharp \Vhatr$ are in general not orthogonal. Re-orthogonalization yields $\bar{\bs V}_{\!\!r} \bar{\bs \Lambda}_r \bar{\bs V}_{\!\!r}^\top := (\bs R^\sharp)^{-1} \hat{\bs R}\mathstrut^\sharp \Vhatr \hat{\bs \Lambda}_r \Vhatr^\top \hat{\bs R} \bs R^{-1}$. Thus, we arrive at an approximation of the form \eqref{eq:pr_precon_approx_prec}, which in turn allows us to employ \eqref{eq:pr_precon_approx_cov} and \eqref{eq:pr_precon_det_ratio_approx} to compute the approximate posterior covariance and the determinant ratio term. The re-orthogonalization, which may be carried out by a randomized eigensolver, requires operations with the prior. Crucially, it does not require any additional applications of $\bs A$ or $\bs A^\sharp$.

However, care must be taken to control the error caused by truncating \eqref{eq:wk_precon_decomp} by choosing $r$. In particular, here we derive a bound on the error \eqref{eq:e1}, while using CG to ensure $e_2 = 0$.
This bound will depend on the choice of the reference prior, which should be taken to be sufficiently weak, as clarified in what follows. 

\begin{theorem} \label{thm:weakest_precon}
    Suppose a reference prior precision $\hat{\bs Q}_\textup{pr} = \hat{\bs R}\mathstrut^\sharp \hat{\bs R}$ satisfying
    \begin{equation}\label{eq:R-cond}
        \hat{\bs R}\mathstrut^\top \hat{\bs R} \preceq
        \bs R^\top \bs R
    \end{equation}
    for every $\bs \theta$ in the support of the hyperprior $\pi_\text{hyp}(\bs \theta)$, and define $\sigma_\textup{min}$ and $\delta_\textup{min}$ such that $\sigma \geq \sigma_\textup{min}$ and $\delta \geq \delta_\textup{min}$ for every such $\bs \theta$.
    Then, for $\Vhatr$ and $\hat{\bs \Lambda}_r = \textup{diag}(\hat{\lambda}_1, \ldots, \hat{\lambda}_r)$ defined as in \eqref{eq:wk_precon_decomp}, using the approximate posterior precision \eqref{eq:wk_precon_approx_prec} results in the error bound
    \begin{equation}\label{eq:wk_precon_error}
        e_1(r,\bs \theta) 
        \leq \frac{1}{2 \sigma^2 \delta^2} \left( r \hat{\lambda}_{r+1} + \sum_{r+1}^n \hat{\lambda}_i \right)
        \leq \frac{1}{2 \sigma_\textup{min}^2 \delta_\textup{min}^2} \left( r \hat{\lambda}_{r+1} + \sum_{r+1}^n \hat{\lambda}_i \right)
    \end{equation}
    for all $\bs \theta$ in the support of the hyperprior.
\end{theorem}
\begin{proof}
Letting $\bs B := (\bs R^\sharp)^{-1} \hat{\bs R}\mathstrut^\sharp$, we denote by $\bar{\lambda}_i$ the $i$-th eigenvalue of $\bs B \Vhatr \hat{\bs \Lambda}_r \Vhatr^\top \bs B^\top$, and by $\lambda_i$ the $i$-th eigenvalue of the PP update, as in \eqref{eq:pr_precon_approx}. Using \eqref{eq:pr_precon_det_ratio} and \eqref{eq:pr_precon_det_ratio_approx},
\begin{align*}
    e_1(r,\bs \theta) 
    &= -  \frac{1}{2} \log \frac{|\bs Q_\text{pr}|}{|\bs Q_{\text{post}}|}
    + \frac{1}{2} \log \frac{|\bs Q_\text{pr}|}{|\hat{\bs Q}_{\text{post},r}|}\\ 
    &= \frac{1}{2} \sum_{i=1}^n \log \left( 1 + \frac{\lambda_i}{\sigma^2 \delta^2} \right) 
    - \frac{1}{2} \sum_{i=1}^r \log \left( 1 + \frac{\bar{\lambda}_i}{\sigma^2 \delta^2} \right)\\
    &= \frac{1}{2} \sum_{i=1}^r \log \left( 1 + \frac{{\lambda}_i - \bar\lambda_i}{\sigma^2 \delta^2 + \bar{\lambda}_i} \right)
    + \frac{1}{2} \sum_{i=r+1}^n \log \left( 1 + \frac{\lambda_i}{\sigma^2 \delta^2} \right).
\end{align*}
To bound the first sum, we seek a bound on $\lambda_i - 
\bar{\lambda}_i$. Note that $\bs B \hat{\bs V} \hat{\bs \Lambda} \hat{\bs V}\mathstrut^\top \bs B^\top = \bs B \Vhatr \hat{\bs \Lambda}_r \Vhatr^\top \bs B^\top + \bs B \Vhatrp \hat{\bs \Lambda}_{r^+} \Vhatrp^\top \bs B^\top$ equals the PP update and therefore has eigenvalues $\lambda_i$. Let $\bs C := \bs B \Vhatr \hat{\bs \Lambda}_r \Vhatr^\top \bs B^\top$ and $\bs D := \bs B \Vhatrp \hat{\bs \Lambda}_{r^+} \Vhatrp^\top \bs B^\top$, then Weyl's inequality for Hermitian matrices states that
\begin{equation*}
    \lambda_i(\bs C) + \lambda_\textup{min}(\bs D) \leq \lambda_i(\bs C + \bs D) \leq \lambda_i(\bs C) + \lambda_\textup{max}(\bs D).
\end{equation*}
Observe that $\lambda_\textup{max}(\bs D) \leq \lambda_\textup{max}(\Vhatrp \hat{\bs \Lambda}_{r^+} \Vhatrp^\top) \lambda_\textup{max}(\bs B^T \bs B)$. From \eqref{eq:R-cond}, we have
$    \bs B^\top \bs B 
    \preceq \bs I
    $,
so $\lambda_\textup{max}(\bs D) \leq \hat{\lambda}_{r+1}$. Since positive semi-definiteness implies $\lambda_\textup{min}(\bs D) \geq 0$, Weyl's inequality becomes
\begin{equation*}
    \bar{\lambda}_i \leq \lambda_i \leq \bar{\lambda}_i + \hat{\lambda}_{r+1}.
\end{equation*}
Thus the first sum in $e_1$ is bounded by
\begin{equation*}
    \frac{1}{2} \sum_{i=1}^r \log \left( 1 + \frac{\lambda_i - \bar{\lambda}_i}{\sigma^2 \delta^2 + \bar{\lambda}_i} \right) 
    \leq \frac{1}{2} \sum_{i=1}^r \log \left( 1 + \frac{\hat{\lambda}_{r+1}}{\sigma^2 \delta^2} \right) 
    \leq \frac{1}{2} \frac{r}{\sigma^2 \delta^2} \hat{\lambda}_{r+1},
\end{equation*}
using $\log(1+x) \leq x$ and $\bar{\lambda}_i \geq 0$.

For the second sum in $e_1$, note that $\hat{\bs \Lambda}\mathstrut^{1/2} \hat{\bs V}\mathstrut^\top \bs B^\top \bs B \hat{\bs V} \hat{\bs \Lambda}\mathstrut^{1/2}$ has eigenvalues $\lambda_i$, by the cyclic property of eigenvalues. Since $\bs B^\top \bs B \preceq \bs I$, we have $\hat{\bs \Lambda}\mathstrut^{1/2} \hat{\bs V}\mathstrut^\top \bs B^\top \bs B \hat{\bs V} \hat{\bs \Lambda}\mathstrut^{1/2} \preceq \hat{\bs \Lambda}$; thus, again by Weyl's monotonicity theorem, for every $i$, $\lambda_i \leq \hat{\lambda}_i$. Substituting $\hat{\lambda}_i$ for $\lambda_i$ in the second sum of $e_1$ and combining this with the previously derived bound on the first sum gives the first inequality in \eqref{eq:wk_precon_error}. The second inequality follows directly from $\sigma \geq \sigma_\textup{min}$ and $\delta \geq \delta_\textup{min}$.
\end{proof}

This theorem allows for a direct comparison between the error bound \eqref{eq:wk_precon_error} and \eqref{eq:pr_precon_det_ratio_error} for the direct PP update.
We observe that \eqref{eq:wk_precon_error} uses the eigenvalues $\hat\lambda_i$ from the reference prior and the minimal values of $\sigma, \delta$, trading off a lower $\bs \theta$-dependent bound for a higher one that is uniform  for all $\bs \theta$. Note that the additional $r \hat{\lambda}_{r+1}$ term may be significant, but is guaranteed to decrease with $r$ since the eigenvalues $\hat{\lambda}_i$ decrease faster than $1/i$, as eigenvalues of a trace-class operator.
In particular, as in \cref{sec:prior-precon-update}, \eqref{eq:wk_precon_error} suggests a natural cutoff for the rank $r$: the error is small when $\hat{\lambda}_{r+1} \ll \sigma_\text{min}^2 \delta_\text{min}^2$. 
Moreover, it also implies a tighter bound for each value of $\bs \theta$, under which the error is low if $\hat{\lambda}_{r+1} \ll \sigma^2 \delta^2$. This suggests the following procedure: precompute the approximation to a rank $r$ that satisfies the $\mathcal{O}(\sigma_\text{min}^2 \delta_\text{min}^2)$ eigenvalue cutoff, and truncate the approximation further using a $\mathcal{O}(\sigma^2 \delta^2)$-based cutoff for each $\bs\theta$. 

We refer to the approximation in \eqref{eq:wk_precon_approx_prec} as the \emph{weakest prior preconditioned (WP) approximation} since \eqref{eq:R-cond} states that the reference prior $\hat{\bs Q}_\text{pr}$ has ``smaller" precision than $\bs Q_\text{pr}(\bs \theta)$ for any $\bs \theta$ (after normalizing by removing any overall precision scaling $\delta^2$).
The error bound in \cref{thm:weakest_precon} requires a $\hat{\bs Q}_\text{pr}$ that satisfies \eqref{eq:R-cond} for every $\bs \theta$. In particular,
the hyperprior must not permit arbitrarily weak priors, arbitrarily low noise variance, or arbitrarily high prior variance ($\sigma, \delta = 0$). These issues are examined in more detail within a concrete example in \cref{sec:adv_diff}.

As an example of a reference prior, we consider
a discretized Mat\'ern kernel Gaussian random field, a common choice for PDE-governed inverse problems. In particular, we consider kernels with covariance operators given by the operator $\delta^{-2} (I - \gamma \Delta)^{-2}$, which are Mat\'ern with $\nu = 2 - d/2$ in dimension $d$ \cite{lindgren_explicit_2011}. Thus the action of the covariance matrix corresponds to two successive solutions of a PDE with operator $I - \gamma \Delta$ discretized using finite elements. Letting $\delta,\gamma > 0$ be hyperparameters in $\bs \theta$, we state the following lemma, which allows us to apply \cref{thm:weakest_precon}.
\begin{lemma}
\label{lem:matern}
    Let $Q_\textup{pr}(\bs \theta) = \delta^{2} (I - \gamma \Delta)^{2}$, and $\gamma_\text{min} \geq 0$ be the minimum $\gamma$ in the support of the hyperprior, and define a reference prior with precision $\hat{Q}_\textup{pr} = (I - \gamma_\text{min} \Delta)^{2}$. Then the discretizations $\bs Q_\textup{pr}(\bs \theta)$ and $\hat{\bs Q}_\textup{pr}$ satisfy \eqref{eq:R-cond} for every $\bs \theta$ in the support of the hyperprior.
\end{lemma}
\begin{proof}
    In a finite element space with mass matrix $\bs M$ and discretization $\bs L$ of $-\Delta$, 
    \begin{equation*}
        \bs Q_\textup{pr} = \delta^2 \bs M^{-1} \bs R^\top \bs R = \bs M^{-1} \delta (\bs M + \gamma \bs L) \bs M^{-1} \delta (\bs M + \gamma \bs L)
    \end{equation*}
    and thus $\bs R^\top \bs R = \bs M + 2 \gamma \bs L + \gamma^2 \bs L \bs M^{-1} \bs L$.
    Analogously, $\hat{\bs R}\mathstrut^\top \hat{\bs R} = \bs M + 2 \gamma_\text{min} \bs L + \gamma_\text{min}^2 \bs L \bs M^{-1} \bs L$.
    The condition \eqref{eq:R-cond} then follows directly from the fact that $\gamma_\text{min} \leq \gamma$ for every $\bs \theta$ in the support of the hyperprior, and that $\bs M$ and $\bs L$ are both symmetric positive semi-definite.
\end{proof}

\subsection{Unpreconditioned precision update (UP)}\label{sec:unprecond}
The arguably simplest approach to the problem of repeated computation is to remove the prior preconditioning entirely. As a consequence, the rank required for an adequate approximation is likely higher and may only be bounded by the number of observations. Nonetheless, in cases where there is no natural candidate for the weakest reference prior, removing the preconditioning may be a necessary alternative.

A direct unpreconditioned approximation $\bs A^\sharp \bs P \bs A \approx \Wr \bs \Lambda_r \Wr^\diamond$ can be seen as the limit of the WP approach as the reference prior becomes weaker. In particular, it is equivalent to preconditioning with a reference ``prior" with precision $\hat{\bs Q}_\text{pr} = \bs I$, such that $\hat{\bs R}\mathstrut^\top \hat{\bs R} = \bs M$, since
\begin{equation*} 
    (\hat{\bs R}\mathstrut^\sharp)^{-1} \bs A^\sharp \bs P \bs A \hat{\bs R}\mathstrut^{-1}
    = \hat{\bs V} \hat{\bs \Lambda} \hat{\bs V}\mathstrut^\top
    \iff 
    \bs A^\sharp \bs P \bs A 
    = (\hat{\bs R}\mathstrut^{-1} \hat{\bs V}) \hat{\bs \Lambda} (\hat{\bs R}\mathstrut^{-1} \hat{\bs V} )^\diamond
    = \bs W \hat{\bs \Lambda} \bs W^\diamond,
\end{equation*}
where $\bs W := \hat{\bs R}\mathstrut^{-1} \hat{\bs V}$. Although the identity $\bs I$ is not a valid prior precision in the infinite dimensional limit, the procedure for approximating the posterior and its use in evaluating $\pi(\bs \theta |\bs y)$ remain unchanged. In particular, the covariance and log determinant ratio can be computed as in \eqref{eq:pr_precon_approx_cov} and \eqref{eq:pr_precon_det_ratio_approx}. 

This equivalence suggests that if \cref{eq:R-cond} is satisfied, we can use \cref{thm:weakest_precon} to estimate the rank $r$ needed for a given level of accuracy in the posterior covariance approximation. In the Mat\'ern kernel example, \cref{eq:R-cond} follows trivially from \cref{lem:matern} since $\hat{\bs Q}_\text{pr} = \bs I$ is exactly the Mat\'ern prior precision with $\gamma = 0$. More generally, in finite dimensions, it is always possible to choose $\delta$ in \eqref{eq:pr_def} to be sufficiently small such that the condition is satisfied, but doing so could arbitrarily inflate the rank $r$ at which $\hat{\lambda}_{r+1} \ll \sigma_\text{min}^2 \delta_\text{min}^2$. To avoid this issue, for general priors we define $\delta^2 = \lambda_\text{min}(\bs Q_\text{pr})$, the smallest eigenvalue of $\bs Q_\text{pr}$, so that $\lambda_\text{min}(\bs R^\sharp \bs R) = 1$. 
This choice of $\delta^2$ remains bounded for any discretization of an infinite-dimensional prior $Q_\text{pr}$ and will, in fact, converge as
$\lambda_\text{min}(Q_\text{pr})>0$ as the discretization is refined.
This follows since the covariance operator $Q_\text{pr}^{-1}$ is trace-class.
With this definition of $\delta$, we can apply \cref{thm:weakest_precon} since any $\bs R^\sharp \bs R$ has eigenvalues $\geq 1$, so $\bs I \preceq \bs R^\sharp \bs R$ in $\R^{n \times n}_M$ and $\hat{\bs R}\mathstrut^\top \hat{\bs R} = \bs M \preceq \bs R^\top \bs R$ in $\R^{n \times n}$.

\subsection{Comparison of computational complexity} \label{sec:complexity-comparison}
We assumed in \cref{sec:complexity} that the computational cost of evaluating $\pi(\bs \theta|\bs y)$ is primarily dominated by forward/adjoint evaluations of the parameter-to-observable maps $\bs A$ and $\bs A^\sharp$, and secondarily by solves with, and applications of, the prior precision $\bs Q_\text{pr}$.
In \cref{tab:complexity}, we show a comparison of the number of operations needed to compute $\pi(\bs \theta|\bs y)$ for $N$ different hyperparameters $\bs \theta$ using the PP approach from \cref{sec:prior-precon-update}, the WP approach from \cref{sec:weakest}, and the UP approach from \cref{sec:unprecond}.
We treat a solve or application of the reference prior $\hat{\bs Q}_\text{pr}$ as equivalent to one with $\bs Q_\text{pr}$, because a natural way to define a reference prior is to take $\bs Q_\text{pr}$ for a specific choice of $\bs \theta$, as in the Mat\'ern GRF with $\gamma_\text{min}$ from \cref{lem:matern}. In addition to these counts, we include in the fourth row of \cref{tab:complexity} the cost of Gram-Schmidt orthogonalization, which is required in the randomized eigensolver. This step is the only operation with complexity $O(r^2)$ and therefore has
comparable cost as prior precision solves and applications. For simplicity, we neglect any reduction in $r$ performed in the online phase of WP and UP, and instead treat $r$ as fixed.
\begin{table}[t]
    \caption{Complexity comparison of PP, WP and UP methods from \cref{sec:prior-precon-update,sec:weakest,sec:unprecond} for $N$ evaluations of $\pi(\bs \theta|\bs y)$. The complexity is given in terms of forward and adjoint PDE solves (first row), solves with the prior or reference prior precision (second row), applies of the precision (third row), and Gram-Schmidt cost (fourth row). To emphasize that the required ranks of the approximations differ, they are denoted by $r$, $\hat{r}$, and $\bar{r}$ for each of the methods, respectively. In our implementation, the padding fraction $p$ is $1/2$ and the number of CG iterations $\ell$ until convergence is typically 2 or 3.}
    \centering
    \resizebox{\textwidth}{!}{%
    \begin{tabular}{c|c|c|c}
         & PP & WP & UP \\
        \hline
        $\bs A, \bs A^\sharp$ applies & $N(2(1+p)r + 2 \ell) + 1$ & $2 N \ell + 2(1+p)\hat{r} + 1$ & $2 N \ell + 2(1+p)\bar{r} + 1$ \\
        \hline
        $\bs Q_\text{pr}$ solves & $N((1+p)r + \ell)$ & $N(\hat{r}+\ell) + (1+p)\hat{r}$ & $N(\bar{r} + \ell)$ \\
        \hline
        $\bs Q_\text{pr}$ applies & $N(3(1+p)r + \ell + 1)$ & $N((4+3p)\hat{r} + \ell + 1) + 3(1+p)\hat{r}$ & $N(3(1+p)\bar{r} + \ell + 1)$ \\
        \hline
        G-S cost & $O(Nr^2)$ & $O((N+1)\hat{r}^2)$ & $O((N+1)\bar{r}^2)$ \\
        \hline
    \end{tabular}}
    \label{tab:complexity}
\end{table}

Two steps of the computation account for all the costly operations: (1) finding the low rank approximation, and (2) solving for the posterior mean $\bs \mu_\text{post}$.
In the first stage of each method, we find an approximation of the form \eqref{eq:pr_precon_approx} by solving a generalized eigenvalue problem.For a target rank $r$ and an oversampling fraction $p$ (here set to $1/2$), the algorithm outputs the leading $r$ eigenvalues and eigenvectors. This requires $(1+p)r$ applications of $\bs A^\sharp \bs P \bs A$ to random vectors, which corresponds to $2(1+p)r$ applications of $\bs A$ or $\bs A^\sharp$, as well as $(1+p)r$ linear solves with $\bs Q_\text{pr}$ or $\hat{\bs Q}_\text{pr}$. The generalized eigenvectors are orthogonalized in the $\bs Q_\text{pr}$-weighted inner product using a modified Gram-Schmidt method, \cite[Alg.~3]{saibaba2016randomized}, which performs an additional $3(1+p)r$ applications of $\bs Q_\text{pr}$ or $\hat{\bs Q}_\text{pr}$ and $O(r^2)$ inner products. In PP, this computation is performed for each $\bs \theta$, while in WP and UP, it is performed once as an ``offline'' precomputation and for each $\bs\theta$ during the re-orthogonalization. In UP, there are no prior solves and applies in the offline stage as $\hat{\bs Q}_\text{pr} = \bs I$.  WP and UP require the additional step of reorthogonalizing the columns of $\bs B \Vhatr$ with respect to $\bs Q_\text{pr}(\bs \theta)$ for each $\bs \theta$. This does not require any additional applications of $\bs A$ or $\bs A^\sharp$. The orthogonalization requires applying $\hat{\bs Q}_\text{pr}$ to $r$ vectors and then solving a generalized eigenvalue problem, which in total requires $(1+p)r$ linear solves, $3(1+p)r$ applications of $\bs Q_\text{pr}$, and $O(r^2)$ inner products.

The second costly stage of each method is the computation of the posterior mean $\bs \mu_\text{post} = \bs Q_\text{post}^{-1} (\bs Q_\text{pr} \bs \mu_\text{pr} + \bs A^\sharp \bs Q_{\bs \varepsilon} \bs y)$ for each $\bs\theta$. Substituting in the approximation $\bs Q_{\text{post},r}^{-1}$ introduces a nonzero error $e_2(r,\bs \theta)$, which, as previously discussed, may be difficult to bound a priori. Thus, generally we use a CG solver with $\bs Q_{\text{post},r}$ as a preconditioner. 
Each CG iteration requires one matrix-vector product with the exact posterior precision, amounting to the application of $\bs A$, $\bs A^\sharp$, and $\bs Q_\text{pr}$, and one solve with $\bs Q_\text{pr}$ in the preconditioner. Finally, computing $\bs Q_\text{pr} \bs \mu_\text{pr} + \bs A^\sharp \bs Q_{\bs \varepsilon} \bs y$ requires one application of $\bs Q_\text{pr}$ and one of $\bs A^\sharp$, the latter of which can be precomputed.

An overview of the total operation counts is given in \cref{tab:complexity}, where $\ell$ is the number of CG iterations. Importantly, although all methods use $O(N)$ forward and adjoint applications, WP and UP remove the leading $O(Nr)$ term that arises from repeatedly recomputing the low-rank approximation. For each method, the number of $\bs Q_\text{pr}$ solves and applications is also $O(Nr)$, with WP and UP requiring slightly fewer solves than PP, but WP needing somewhat more applications. Although the ranks for WP and UP are expected to be larger for a fixed target accuracy, for moderate to large $N$ both are more efficient than PP thanks to the precomputation of the low-rank approximation. 
In cases where $e_2(r,\bs \theta)$ is small even when $\bs Q_{\text{post},r}$ is used to compute $\bs\mu_\text{post}$ and thus CG iterations are not necessary, the advantage of this amortization is even greater: with $\ell = 0$, all dependence on $N$ in the number of $\bs A, \bs A^\sharp$ applies is removed from WP and UP. An example of such a case is presented in \cref{sec:adv_diff}.

\section{Numerical example: initial condition inference in advection-diffusion}
\label{sec:adv_diff}
We consider the problem of inferring the initial condition of an advection-diffusion equation from later-time observations, with uncertain hyperparameters in the prior and noise covariance. On a domain $\Omega\subset \mathbb R^d$, $d\in \{2,3\}$, let $u(x,t)$ be the solution to
\begin{subequations}\label{eq:adv-diff}
\begin{alignat}{2}
    u_t - \kappa \Delta u + \bs v \cdot \nabla u &= 0 && \text{ in } \: \Omega \times [0,T],\\
    u(x,0) &= m(x)\: && \text{ in } \: \Omega,\\
    \kappa \nabla u \cdot \bs n &= 0 && \text{ on } \: \partial \Omega.
\end{alignat}
\end{subequations}
Here $\bs v$ and $\kappa>0$ denote a given advection velocity field and the diffusion coefficient, and $\bs n$ denotes the unit outward normal vector on the boundary $\partial \Omega$.
The parameter $\bs m$ is the finite element discretization of the initial condition $m(x)$, and the observed data is $\bs y\in \R^q$ with components
\begin{equation*}
    y_i = u(x_i,t_i) + \varepsilon_i, \quad i=1,\ldots,q
\end{equation*}
for observation points and times $(x_i,t_i)$ and additive Gaussian noise $\bs \varepsilon \sim \mathcal N(\bs 0,\sigma^{2} \bs I)$. Because the solution of the advection-diffusion equation depends linearly on its initial condition, the parameter-to-observable map $\bs A$ is also linear (or affine in the case of non-zero right hand sides in \eqref{eq:adv-diff}). The prior distribution on $\bs m$ is given by the discretization of a mean zero Mat\'ern Gaussian random field with covariance operator $\delta^{-2} ( I - \gamma  \Delta )^{-2}$ combined with Robin boundary conditions with a constant, empirically derived coefficient to mitigate boundary artifacts \cite{DaonStadler18, villa_hippylib_2018}. Note that this boundary condition limits the admissible options for the prior mean, because non-zero constant functions generally do not satisfy the Robin boundary condition and therefore do not belong to the Cameron–Martin space. Choosing a zero prior mean circumvents these complications.
The hyperparameter vector is $\bs \theta = [\gamma, \delta, \sigma]^\top$, consisting of the two parameters $\gamma$ and $\delta$, which enter the prior precision and the noise standard deviation $\sigma$, each with an independent uniform hyperprior.

Our implementation uses the hIPPYlib Python package \cite{villa_hippylib_2018}, which builds on the FEniCS finite element library \cite{logg2012automated}. We adapt an example in hIPPYlib to incorporate the hyperparameters and implement the matrix-free linear algebra methods and solvers required for the proposed methods.
Unless otherwise specified, $\Omega$ is discretized using an unstructured triangular or tetrahedral mesh. We use linear $P_1$ finite elements for the initial condition parameter $\bs m$, and unless otherwise specified, quadratic $P_2$ finite elements for $u$ in \eqref{eq:adv-diff}. $L^2$-projections are used to relate the initial condition discretization $\bs m$ to the piecewise quadratic initial condition used in \eqref{eq:adv-diff}. For time integration, we apply the backward Euler method with uniformly spaced time steps, and the advection term in \eqref{eq:adv-diff} is stabilized using a Galerkin least squares approach.
For more details on \eqref{eq:adv-diff}, its discretization, and the derivation of the adjoint $\bs A^\sharp$, we refer to \cite{villa_hippylib_2018}.

\subsection{Two-dimensional domain}
\label{sec:2D_adv_diff}
On the left of \cref{fig:dim_indep}, we show the domain $\Omega$ along with the velocity field $\bs v$, obtained by solving the Navier–Stokes equations with no-slip boundary conditions imposed on the top, bottom, and inner boundaries, and unit tangential velocity Dirichlet data on the left and right boundaries in the upward and downward directions, respectively. We set $\kappa=\num{1e-3}$ and place four equally spaced points along each edge of every rectangular cut-out, giving a total of 32 spatial sensor positions. Measurements are taken every 0.4 seconds from $t=2.4$ to $t=4.0$, resulting in the observation vector $\bs y\in \R^{160}$.

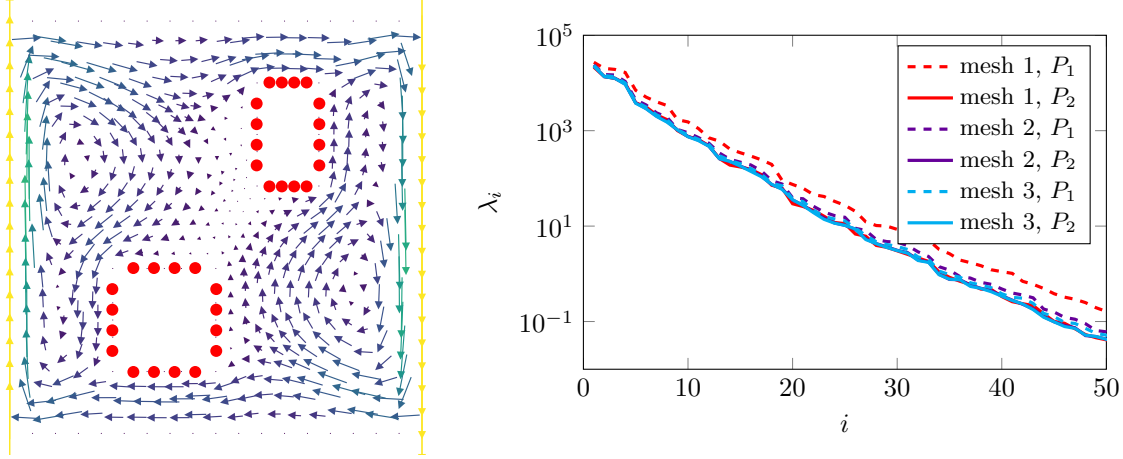
\begin{figure}[t]
\centering
    \begin{minipage}{0.37\textwidth}
    \centering
    \includegraphics[width=\textwidth]{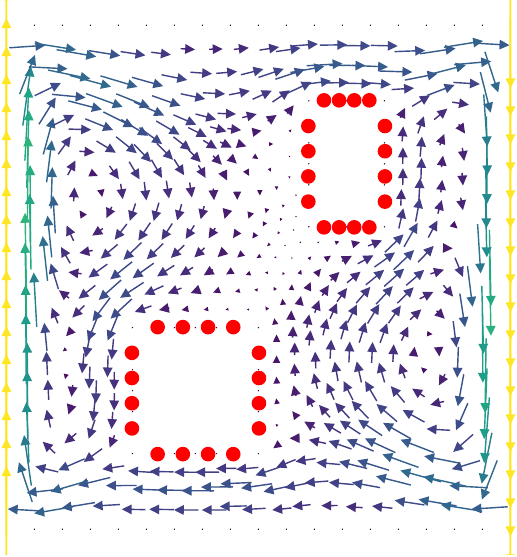}
    \end{minipage}
    \hspace{1ex}
    \begin{minipage}{0.6\textwidth}
    \centering
    \begin{tikzpicture}
  \begin{axis}[
  height=6cm,
  width=8.5cm,
  xmin=0,
  xmax=50,
  ymin=1e-2,
  ymax=1e5,
  ymode=log,
  compat=1.3,
  xlabel={$i$},
  ylabel={$\lambda_i$},
  legend pos=north east,
  legend cell align=left,
  legend style={font=\small}
]  
     \addplot[dashed, color=red, very thick, mark = none]
	table[x expr={\thisrow{r}},y expr={\thisrow{605linear}}] {data/spectra_dim_indep.txt};  
    \addplot[color=red, very thick, mark = none]
	table[x expr={\thisrow{r}},y expr={\thisrow{605quadratic}}] {data/spectra_dim_indep.txt};
    \addplot[dashed, color=red!40!blue, very thick, mark = none]
	table[x expr={\thisrow{r}},y expr={\thisrow{2363linear}}] {data/spectra_dim_indep.txt};  
    \addplot[color=red!40!blue, very thick, mark = none]
	table[x expr={\thisrow{r}},y expr={\thisrow{2363quadratic}}] {data/spectra_dim_indep.txt};
    \addplot[dashed, color=cyan, very thick, mark = none]
	table[x expr={\thisrow{r}},y expr={\thisrow{5443linear}}] {data/spectra_dim_indep.txt};
    \addplot[color=cyan, very thick, mark = none]
	table[x expr={\thisrow{r}},y expr={\thisrow{5443quadratic}}] {data/spectra_dim_indep.txt};
    \legend{{mesh 1, $P_1$}, {mesh 1, $P_2$}, {mesh 2, $P_1$}, {mesh 2, $P_2$}, {mesh 3, $P_1$}, {mesh 3, $P_2$}}; %
    \end{axis}
  \end{tikzpicture}
  \end{minipage}
  \caption{Left: Velocity field $\bs v$ (arrows) and observation points (in red) for the 2D advection-diffusion problem. Right: Spectra of the WP low rank approximation ($\gamma = 0.0015$) for multiple finite element discretizations. Dashed lines indicate that linear finite ($P_1$) elements are used for $u$ in \eqref{eq:adv-diff}, and solid lines indicate quadratic ($P_2$) elements. Meshes 1, 2, and 3 have 1088, 4472, and 10488 elements, respectively.}
  \label{fig:dim_indep}
\end{figure}

For the time discretization, 80 implicit time steps are used. To study different spatial discretizations and numerically verify mesh-converged behavior, we show the spectra of the WP precision update for three successively refined meshes on the right side of \cref{fig:dim_indep}, using either linear ($P_1$) or quadratic ($P_2$) Lagrange elements for \eqref{eq:adv-diff}. The figure indicates that $P_1$ elements introduce some error, likely caused by numerical diffusion, whereas $P_2$ elements on a mesh with 4472 elements yield numerically well converged leading eigenvalues. This justifies our use of $P_2$ elements for \eqref{eq:adv-diff}. For the following experiments, we thus use the mesh consisting of 4472 triangles, resulting in 2363 unknowns for $\bs m$ and 9199 spatial unknowns for $u$ in \eqref{eq:adv-diff}. 

The true initial condition used for generating synthetic observation data is
\begin{equation*}
    m(x_1, x_2) = \min(0.5, \exp(-100 ((x_1-0.35)^2 + (x_2-0.7)^2))),
\end{equation*}
and the true value of the noise standard deviation is $\sigma = 0.01$. \cref{fig:forward} shows snapshots of $u$ at different times. The hyperparameters, which are treated as nuisance parameters in this problem, are assigned independent uniform hyperpriors on the intervals $\gamma \in [0.0015, 10]$, $\delta \in [1, 100]$, $\sigma \in [0.003, 0.1]$.
Next, we study the approximation properties and computational times of the proposed methods, and then demonstrate our procedure for computing the distribution of a quantity of interest in \cref{sec:QoI}.


\begin{figure}[t]
\centering
\begin{tikzpicture}
\node at (-2,2.6) {$t=0$};
\node at (1,2.6) {$t=0.5$};
\node at (4,2.6) {$t=2$};
\node at (7,2.6) {$t=4$};
\node at (-2,1) {\includegraphics[width=0.18\linewidth]{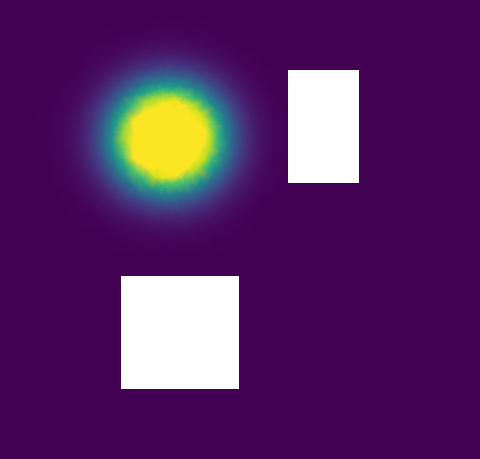}};
\node at (1,1) {\includegraphics[width=0.18\linewidth]{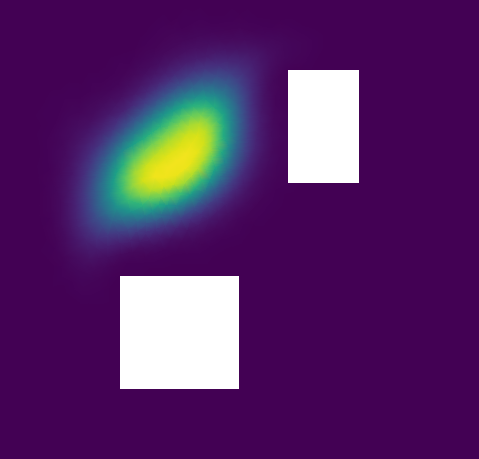}};
\node at (4,1) {\includegraphics[width=0.18\linewidth]{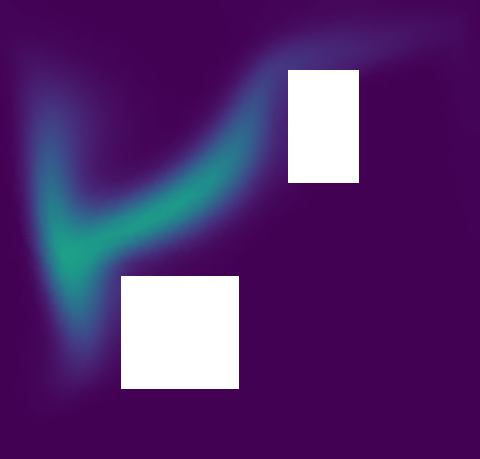}};
\node at (7,1) {\includegraphics[width=0.18\linewidth]{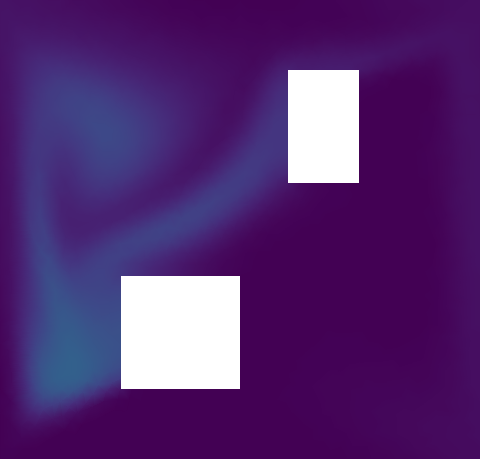}};
\node at (9,1) {\includegraphics[width=0.03\linewidth]{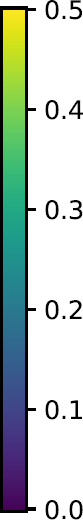}};
\begin{scope}[shift={(-3.35, -0.29)}, x=2.7cm, y=2.6cm]
    \draw[red, very thick] (0.2, 0.7) rectangle (0.4, 0.9);
\end{scope}
\end{tikzpicture}
\label{fig:forward}
\caption{The forward solution of the advection diffusion equation in 2D, at selected times. The quantity of interest $q(\bs m)$ is the average of the initial condition inside the red square.}
\end{figure}
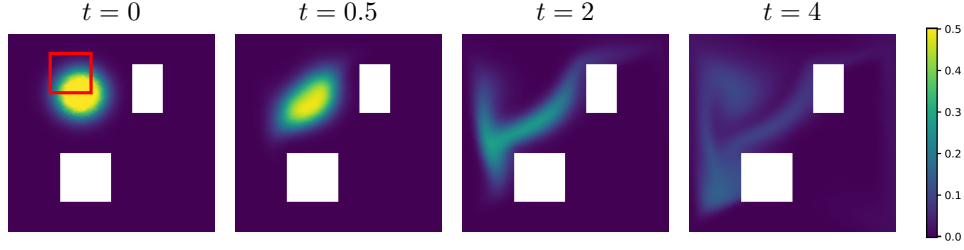

\subsubsection{Comparison of low rank approximations}
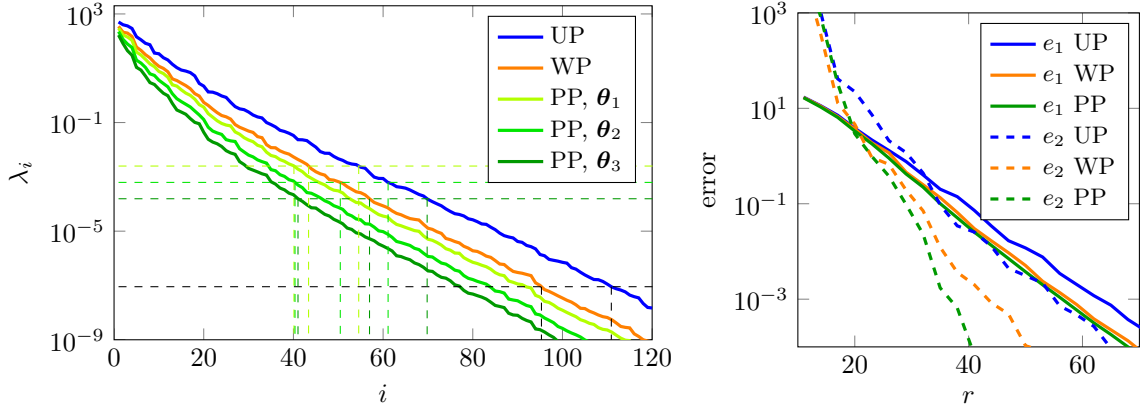
\begin{figure}[t]
\begin{minipage}{0.60\textwidth}
    \centering
    \begin{tikzpicture}
\begin{axis}[
    height=6cm, width=8.7cm,
    xmin=0,
    xmax=120,
    ymin=1e-9,
    ymax=2000,
    compat=1.3,
    xlabel = {$i$},
    ylabel = {$\lambda_i$},
    ymode=log,
    legend pos= north east,
    legend cell align=left,
    legend style={font=\small}
]
    \addplot[color=blue, very thick, mark = none, name path global = unprecon]
	table[x expr={\thisrow{r}},y expr={\thisrow{unprecon}}] {data/spectra.txt};
     \addplot[color=orange, very thick, mark = none, name path global = weakest]
	table[x expr={\thisrow{r}},y expr={\thisrow{weakest}}] {data/spectra.txt};
    \addplot[color=C1, very thick, mark = none, name path global = prior1]
	table[x expr={\thisrow{r}},y expr={\thisrow{prior1}}] {data/spectra.txt};
    \addplot[color=C2, very thick, mark = none, name path global = prior2]
	table[x expr={\thisrow{r}},y expr={\thisrow{prior2}}] {data/spectra.txt};
    \addplot[color=C3, very thick, mark = none, name path global = prior3]
	table[x expr={\thisrow{r}},y expr={\thisrow{prior3}}] {data/spectra.txt};
    \draw [dashed, C1, name path = cutoff_1] (axis cs:1,2.5e-3) -- (axis cs:200,2.5e-3);
    \draw [dashed, C2, name path = cutoff_2] (axis cs:1,6.25e-4) -- (axis cs:200,6.25e-4);
    \draw [dashed, C3, name path = cutoff_3] (axis cs:1,1.5625e-4) -- (axis cs:200,1.5625e-4);
    \draw [dashed, black, name path = cutoff_min] (axis cs:1,9e-8) -- (axis cs:200,9e-8);
    \path [name intersections={of=unprecon and cutoff_min, by=unpre_min}];
    \path [name intersections={of=unprecon and cutoff_1, by=unpre_1}];
    \path [name intersections={of=unprecon and cutoff_2, by=unpre_2}];
    \path [name intersections={of=unprecon and cutoff_3, by=unpre_3}];
    \draw[dashed] (unpre_min) -- (unpre_min |- current axis.south);
    \draw[dashed, C1] (unpre_1) -- (unpre_1 |- current axis.south);
    \draw[dashed, C2] (unpre_2) -- (unpre_2 |- current axis.south);
    \draw[dashed, C3] (unpre_3) -- (unpre_3 |- current axis.south);
    \path [name intersections={of=weakest and cutoff_min, by=weak_min}];
    \path [name intersections={of=weakest and cutoff_1, by=weak_1}];
    \path [name intersections={of=weakest and cutoff_2, by=weak_2}];
    \path [name intersections={of=weakest and cutoff_3, by=weak_3}];
    \draw[dashed] (weak_min) -- (weak_min |- current axis.south);
    \draw[dashed, C1] (weak_1) -- (weak_1 |- current axis.south);
    \draw[dashed, C2] (weak_2) -- (weak_2 |- current axis.south);
    \draw[dashed, C3] (weak_3) -- (weak_3 |- current axis.south);
    \path [name intersections={of=prior1 and cutoff_1, by=pr1_1}];
    \path [name intersections={of=prior2 and cutoff_2, by=pr2_2}];
    \path [name intersections={of=prior3 and cutoff_3, by=pr3_3}];
    \draw[dashed, C1] (pr1_1) -- (pr1_1 |- current axis.south);
    \draw[dashed, C2] (pr2_2) -- (pr2_2 |- current axis.south);
    \draw[dashed, C3] (pr3_3) -- (pr3_3 |- current axis.south);
    \legend{UP, WP, {PP, $\bs \theta_1$}, {PP, $\bs \theta_2$}, {PP, $\bs \theta_3$}};
    \end{axis}
  \end{tikzpicture}
  \end{minipage}
\begin{minipage}{0.39\textwidth}
    \centering
    \begin{tikzpicture}
    \begin{axis}[height=6cm, width=6.1cm,
        xmin=10,
        xmax=70,
        ymin=1e-4,
        ymax=1000,
        compat=1.3,
        xlabel = {$r$},
        ylabel = {error},
        ymode=log,
        legend pos= north east,
        legend cell align = left,
        legend style={font=\small}]   
    \addplot[color=blue, very thick, mark = none, name path global = unprecon]
	table[x expr={\thisrow{r}},y expr={\thisrow{e1_un}}] {data/log_pi_error.txt};
     \addplot[color=orange, very thick, mark = none, name path global = weakest]
	table[x expr={\thisrow{r}},y expr={\thisrow{e1_weak}}] {data/log_pi_error.txt};
    \addplot[color=C3, very thick, mark = none, name path global = prior]
	table[x expr={\thisrow{r}},y expr={\thisrow{e1_pr}}] {data/log_pi_error.txt};
    \addplot[color=blue, dashed, very thick, mark = none, name path global = unprecon2]
	table[x expr={\thisrow{r}},y expr={\thisrow{e2_un}}] {data/log_pi_error.txt};
     \addplot[color=orange, dashed, very thick, mark = none, name path global = weakest2]
	table[x expr={\thisrow{r}},y expr={\thisrow{e2_weak}}] {data/log_pi_error.txt};
    \addplot[color=C3, dashed, very thick, mark = none, name path global = prior2]
	table[x expr={\thisrow{r}},y expr={\thisrow{e2_pr}}] {data/log_pi_error.txt};
    \legend{$e_1$ UP, $e_1$ WP, $e_1$ PP, $e_2$ UP, $e_2$ WP, $e_2$ PP};
    \end{axis}
  \end{tikzpicture}
  \end{minipage}
  \label{fig:spectra_error}
  \caption{Left: Comparison of spectra of each of the three low rank approximations in the 2D advection-diffusion problem. UP (blue) corresponds to $\gamma = 0$, and WP (orange) corresponds to $\gamma_\text{min} = 0.0015$. PP is shown for three representative values, $\bs \theta_1 = [0.003, 50, 0.01]^\top$, $\bs \theta_2 = [0.0075, 25, 0.01]^\top$, and $\bs \theta_3 = [0.015, 12.5, 0.01]^\top$. The dashed black line indicates the cutoff $\lambda_i < 10^{-2} \sigma_\text{min}^2 \delta_\text{min}^2$. The green dashed lines indicate the cutoffs $\lambda_i < 10^{-2} \sigma^2 \delta^2$ for each of the $\bs \theta$'s.
  Right: Errors $e_1(r,\bs \theta)$ (solid) and $e_2(r,\bs \theta)$ (dashed) as a function of rank $r$, for UP (blue), WP (orange) and PP (green) methods, where $\bs \theta = \bs \theta_3$.}
\end{figure}
We now compare the approximation methods PP, WP and UP numerically.
A key distinction between the methods lies in how quickly their eigenvalues decay and the corresponding ranks they require. \cref{fig:spectra_error} (left) shows the spectra of the preconditioned precision update for WP and UP, together with PP spectra for representative values $\bs \theta_1$, $\bs \theta_2$, $\bs \theta_3$ selected from the high probability region of $\pi(\bs \theta|\bs y)$ (see \cref{fig:quadrature}). With $\delta^2$ factored out as in \eqref{eq:pr_def}, and for our particular choice of Mat\'ern kernel prior, $\bs R$ depends only on $\gamma$, so the preconditioned update depends only on the $\gamma$ used in the preconditioning.
The cutoffs from \cref{thm:weakest_precon}, with a tolerance factor of $10^{-2}$, are shown as dashed horizontal lines, with vertical lines indicating the ranks at which the spectra drop below the cutoff. The ranks are also listed in \cref{tab:timing} below.
This figure highlights the advantage gained by preconditioning, with smoother preconditioning priors (higher $\gamma$) corresponding to faster decay in the eigenvalues of the preconditioned update, and therefore lower rank approximations. 

The effects of the faster eigenvalue decay on the accuracy of the approximation are shown on the right in \cref{fig:spectra_error}. Choosing the target value $\bs \theta = \bs \theta_3$, one of the reference points from the left plot, we show the two error components $e_1(r,\bs\theta)$ and $e_2(r,\bs \theta)$ defined in \eqref{eq:e1} and \eqref{eq:e2}
as a function of the approximation rank $r$.
The truth, for the purpose of this error metric, is approximated using PP with CG and rank $r = 250$. Since the update has a rank bounded by the number of observations, 160, this incurs no truncation error. For each method, a decomposition is computed and then truncated to each successive rank to calculate the error. As expected, both errors decrease faster with stronger preconditioning, and slowest with no preconditioning, due to the relative decay rate of the eigenvalues, though the difference in $e_1$ between the methods is small. This indicates that \eqref{eq:wk_precon_error} is a conservative error bound for WP in this problem. Additionally, for moderately large ranks, including those determined by the cutoffs in the left figure, the total error is dominated by $e_1$, indicating that CG is not necessary for an accurate approximation of $\pi( \bs \theta|\bs y)$ in this problem. The choice of $\bs \mu_\text{pr}=0$ is likely significant here, eliminating the first term in $\bs g = \bs Q_\text{pr} \bs \mu_\text{pr} + \bs A^\sharp \bs Q_{\bs \varepsilon} \bs y$, whose norm controls $e_2$.

\cref{tab:timing} displays the ranks for each method derived from \cref{fig:spectra_error}, together with a timing comparison of the three methods. Though CG is not necessary to compute $\bs \mu_\text{post}$ in this problem, $e_2$ may be significant in other problems, particularly if $\bs \mu_\text{pr} \neq \bs 0$. Thus, we present timings with and without CG. The timings include both the required precomputations of the low rank approximation, which is repeated for each $\bs \theta$ in PP, and evaluation.  In principle, the rank could be determined adaptively, but for the purposes of these timings, the approximation was computed to rank $r=50$ and then truncated to the smallest $i$ satisfying the $\lambda_i < 10^{-2} \sigma^2 \delta^2$ cutoff. The fourth row shows a typical value of this rank, for the specific case of $\bs \theta = \bs \theta_3$, compared in the last row to the rank needed to obtain error $e(r,\bs \theta_3)$ below $10^{-2}$. For WP and UP, the low rank approximations were precomputed to the rank determined by the minimum cutoff, shown in the third row, and then truncated for each $\bs \theta$ to the higher cutoff, with typical rank values shown in the fourth row. 
Despite the increase in rank in WP and UP, both achieve a significant speedup relative to PP by avoiding recomputation of the low rank approximation. The speedup is particularly pronounced when no CG is used, since this fully eliminates $\bs A, \bs A^\sharp$ applies from the evaluation phase of WP and UP. While UP avoids some of the applies and solves with $\bs Q_\text{pr}$ required by WP (see \cref{tab:complexity}), the higher rank of the approximation ultimately results in a slightly slower method. Thus, when a ``weakest'' reference prior is available, WP is the fastest.

\begin{table}[t]
    \caption{Timing and rank comparison for PP, WP and UP applied to the 2D advection-diffusion problem. The first row shows time in seconds for $N=100$ evaluations of $\pi(\bs \theta|\bs y)$, and the second row shows the same time if no CG is used in computing $\bs \mu_\textup{post}$. The remaining rows show the ranks used in each method, for varying choices of cutoff. In the third row, the rank is the minimum $i$ such that $\lambda_i < 10^{-2} \sigma_\textup{min}^2 \delta_\textup{min}^2$. No value is shown for PP, since $\bs \theta$ is known and the minimum cutoff is not used. The fourth row repeats this for $\lambda_i < 10^{-2} \sigma^2 \delta^2$, using $\bs \theta_3$ as a target value. The last row shows the smallest $r$ such that $e(r,\bs \theta_3) < 10^{-2}$.
    Timings were performed on a 2026 laptop equipped with an Intel Core Ultra 7 processor with 8 cores and 16 GB of RAM.}
    \centering
    \begin{tabular}{c|c|c|c}
         & PP & WP & UP \\
         \hline
         time for $N=100$ & 870 s & 79 s & 86 s \\
         \hline
         time for $N=100$, no CG & 806 s & 45 s & 57 s \\
        \hline \hline
        rank ($\sigma^2_\text{min}\delta^2_\text{min}$) & --- & 95 & 110 \\
        \hline
        rank ($\sigma^2\delta^2$) & 41 & 56 & 69 \\
        \hline
        rank (error) & 47 & 48 & 54 \\
        \hline
    \end{tabular}
    \label{tab:timing}
    
\end{table}

\subsubsection{Evaluation of the quantity of interest marginal}\label{sec:QoI}
Next, we focus on the computation of the distribution of a linear quantity of interest (QoI) using $\pi(\bs \theta|\bs y)$. 
The QoI $q(\bs m)$ we consider is the average of $m(x)$ in the square region $[0.2, 0.4]\times [0.7,0.9]$, indicated by the red square in \cref{fig:forward}. The goal is to compute the QoI distribution $\pi(q|\bs y)$, which is 
given by 
\begin{equation*}
\pi(q|\bs y) = \int \pi(q|\bs \theta, \bs y) \pi(\bs \theta|\bs y) \, d \bs \theta.
\end{equation*}
Since this is a three-dimensional integral over a unimodal distribution, it can be approximated well by simple quadrature. Following the approach of \cite{rue2009approximate}, we (1) find $\bs \theta^\ast = \arg \max_{\bs \theta} \pi(\bs \theta|\bs y)$ and then (2) find a Gaussian approximation at $\bs \theta^\ast$ using finite differencing. Then (3), we define a grid around $\bs \theta^\ast$ by taking steps in each direction scaled by the approximate standard deviation until $\pi(\bs \theta|\bs y)$ drops below a threshold, and (4) check each point in the grid and keep only those exceeding the probability threshold. Applying this procedure, which is commonly used in INLA, to the 2D advection-diffusion problem with steps of size one standard deviation and a probability threshold of 2.5 Gaussian standard deviations yields 39 such quadrature points. \cref{fig:quadrature} (left) shows contours of $\pi(\bs \theta|\bs y)$ with the quadrature points superimposed. This procedure evaluated $\pi(\bs \theta|\bs y)$ 231 times in the Nelder-Mead optimization (with starting point $\bs \theta_0 = [1, 10, 0.05]^\top$) and 75 times to find the quadrature points, so at least for $\bs \theta \in \R^3$, quadrature only a modest amount of evaluations beyond what empirical Bayes already requires.

Since the quantity of interest is linear and $\pi(\bs m | \bs \theta,\bs y)$ is Gaussian, $\bs \pi(q |\bs \theta, \bs y)$ is also Gaussian, so quadrature amounts to approximating the marginalized $\pi(q|\bs y)$ by a mixture of Gaussians. The outcome is shown in \cref{fig:quadrature} (right), compared to the empirical Bayes estimator $\pi(q|\bs \theta^\ast, \bs y)$ and two other conditional distributions with reference values of $\bs \theta$.
\begin{figure}[t]
    \begin{tikzpicture}
    \node at (0,0) {\includegraphics[width=.33\textwidth]{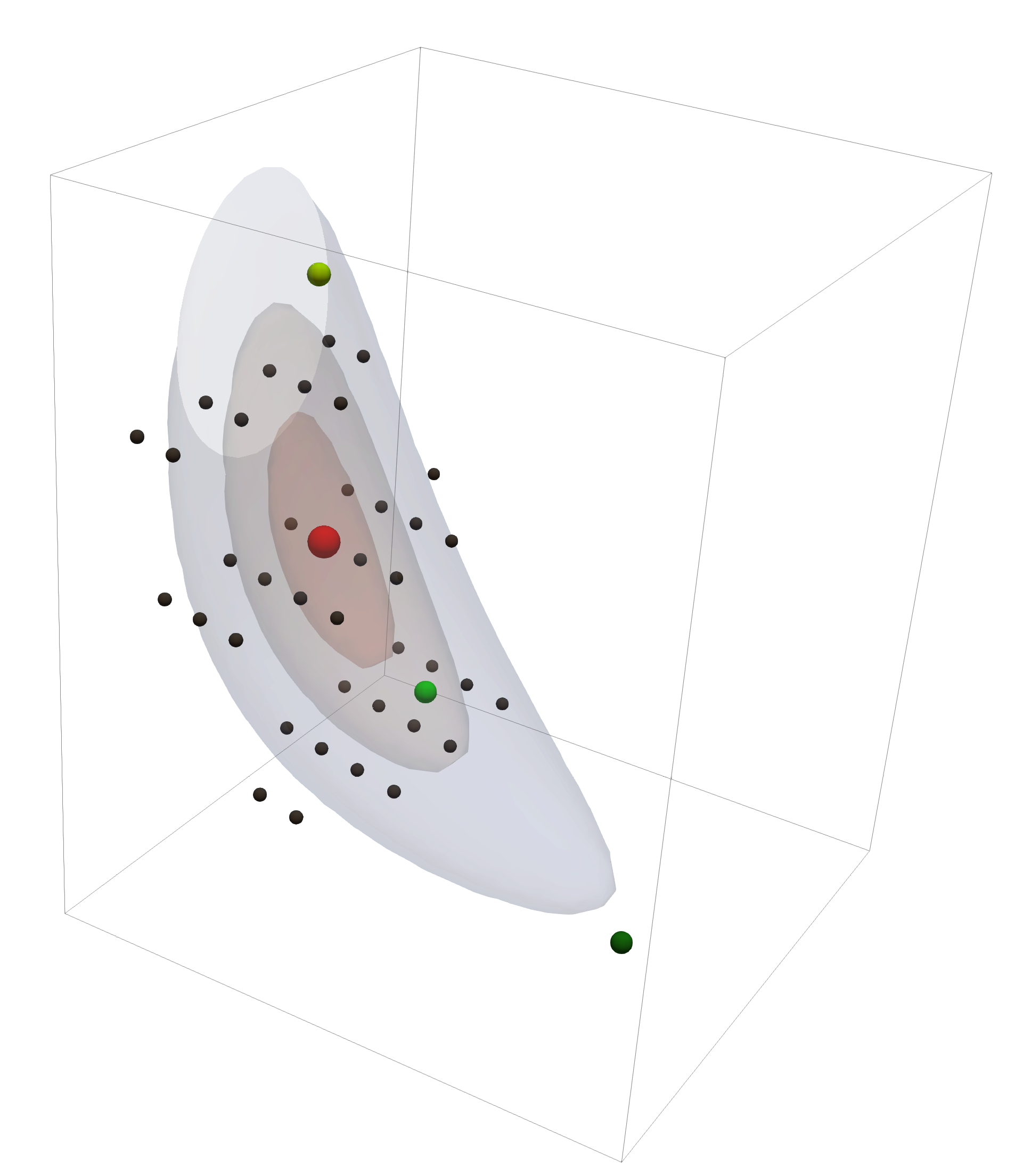}};    
    \coordinate (A) at (-2.25,2.01);
    \coordinate (B) at (-2.18,-1.59);
    \coordinate (C) at (0.54,-2.8);
    \coordinate (D) at (1.75,-1.29);
    \pgfmathsetmacro{\eps}{0.1}
    \node at (-3.1,0) {$\delta$};
    \pgfmathsetmacro{\mindel}{15}
    \pgfmathsetmacro{\maxdel}{60}
    \draw[black!80!blue] (B)--(A);
    \foreach \delta in {0.111111,0.5,0.888888}{
        \pgfmathsetmacro{\tickvalue}{(\maxdel-\mindel)*\delta+\mindel}
        \draw[black!80!blue]  ($(B)!\delta!(A) + \eps*(1,0)$) -- ($(B)!\delta!(A) - \eps*(1,0)$) node[anchor=east, font=\tiny] {\pgfmathprintnumber{\tickvalue}};
    }
    \node at (-1.4,-2.7) {$\gamma$};
    \pgfmathsetmacro{\mingam}{0.0015}
    \pgfmathsetmacro{\maxgam}{0.02}
    \draw[black!80!blue] (B)--(C);
    \foreach \gamma in {0.081081081,0.459459459,0.837837838}{
        \pgfmathsetmacro{\tickvalue}{(\maxgam-\mingam)*\gamma+\mingam}
        \draw[black!80!blue]  ($(B)!\gamma!(C) + \eps*(0.707,0.707)$) -- ($(B)!\gamma!(C) - \eps*(0.707,0.707)$) node[anchor=east, font=\tiny] {\pgfmathprintnumber[fixed, precision=4]{\tickvalue}};
    }
    \node at (1.9,-2.5) {$\sigma$};
    \pgfmathsetmacro{\minsig}{0.00085}
    \pgfmathsetmacro{\maxsig}{0.00120}
    \draw[black!80!blue] (C)--(D);
    \foreach \sigma in {1/7,0.44,5/7}{
        \pgfmathsetmacro{\tickvalue}{(\maxsig-\minsig)*\sigma+\minsig}
        \draw[black!80!blue]  ($(C)!\sigma!(D) + \eps*(-0.866,0.5)$) -- ($(C)!\sigma!(D) - \eps*(-0.866,0.5)$) node[anchor=west, font=\tiny] {\pgfmathprintnumber[fixed, precision=5]{\tickvalue}};
    }
    \end{tikzpicture}
\hspace{.4cm}
    \begin{tikzpicture}
    \begin{axis}[height=6cm, width=9cm,
        xmin=0.1,
        xmax=0.275,
        ymin=-1,
        ymax=14,
        compat=1.3,
        xlabel = {$q$},
        ylabel = {$\pi(q)$},
        legend pos= north west]   
     \addplot[color=C1, very thick, mark = none]
	table[x expr={\thisrow{q}},y expr={\thisrow{theta_1}}] {data/piQoI.txt};
    \addplot[color=C3, very thick, mark = none]
	table[x expr={\thisrow{q}},y expr={\thisrow{theta_3}}] {data/piQoI.txt};
    \addplot[color=red, very thick, mark = none]
	table[x expr={\thisrow{q}},y expr={\thisrow{theta_opt}}] {data/piQoI.txt};
    \addplot[color=black, very thick, mark = none]
	table[x expr={\thisrow{q}},y expr={\thisrow{marginalized}}] {data/piQoI.txt};  
    \legend{$\bs \theta_1$, $\bs \theta_3$, $\bs \theta^\ast$, $\pi(q|\bs y)$};
    \draw [dashed, black] (axis cs:0.2320159,-1) -- (axis cs:0.2320159,14); 
    \end{axis}
  \end{tikzpicture}
  \caption{Left: Contours of $\pi(\bs \theta|\bs y)$. Quadrature points are shown in black, $\bs \theta^\ast = \arg \max_{\bs \theta} \pi(\bs \theta|\bs y)=[0.00421, 34.1, 0.0102]^\top$ is shown in red, and the reference points $\bs \theta_1, \bs \theta_2$ and $\bs \theta_3$ are shown in light, medium, and dark green, respectively.
  Right: Posterior marginal and conditional distributions of a quantity of interest $q$. 
  The marginal $\pi(q|\bs y)$, evaluated by quadrature, is shown in black. The conditional distribution $\pi(q|\bs \theta^\ast, \bs y)$ is in red, and conditional distributions for $\bs \theta_1$ and $\bs \theta_3$ in light and dark green, corresponding to the points on the left. The true value of $q$ is marked by a vertical dashed line.}
  \label{fig:quadrature}
\end{figure}
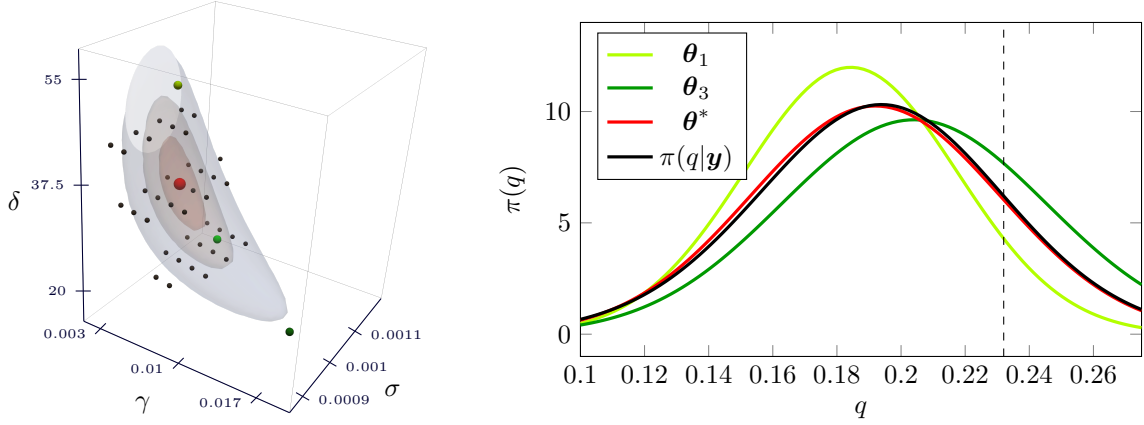
The left panel of \cref{fig:quadrature} shows the hyperparameter marginal $\pi(\bs\theta|\bs y)$ obtained from data $\bs y$.
It is centered near the true noise standard deviation $\sigma= 0.01$, so the noise level is well recovered from the data. The curved shape of the distribution suggests an inversely proportional relationship between $\gamma$ and $\delta$, i.e., that the coefficient of the Laplacian in the prior precision, $\gamma \delta$, is easier to infer than the values of $\gamma$ and $\delta$ separately. The quantity of interest distributions on the right suggest that $\pi(q|\bs \theta, \bs y)$ shows some variation when conditioned on different values of $\bs \theta$ sampled from $\pi(\bs \theta|\bs y)$, and that the marginalized distribution differs visibly, though modestly, from the empirical Bayes estimate $\pi(q|\bs \theta^\ast, \bs y)$. All distributions, conditional and marginal, have a substantial variance reflecting low confidence in the value of the quantity of interest, a result of the problem's ill-posedness.

\subsection{Three-dimensional domain}
Next, to study the scalability of our methods, we apply them to the inverse problem governed by \eqref{eq:adv-diff} over a three-dimensional spatial domain. 
The domain $\Omega$ is a cube with three cutouts representing buildings in a very simple model of a city, shown in \cref{fig:forward3D}. The velocity field is obtained by solving the Navier-Stokes equations with no-slip boundary conditions on the building surfaces, free slip boundary conditions on the top, bottom, front, and back faces of the cube, and diagonal $(0,1,1)$ and $(0,-1,-1)$ velocities on the left and right faces, respectively. The diffusion coefficient is $\kappa = 3 \times 10^{-3}$, and the measurements are spaced in time from 2.4 to 4 time units, as in the 2D problem. In space, there are 51 observation points spread across the buildings: at the center of the top face, at the top four corners, halfway down the side edges, and in the center of the upper and lower halves of each side face for each building. Thus, the observation data is $\bs y \in \R^{255}$. The domain is discretized using 34,916 tetrahedra. Using $P_1$ elements for the initial condition $\bs m$ results in 7,480 unknowns, and using quadratic elements for $u$ in 53,558 spatial unknowns. Again, we solve \eqref{eq:adv-diff} using the backward Euler method with 80 time steps.
The initial condition used to generate synthetic data is
\begin{equation*}
    m(x_1, x_2, x_3) = \min(0.5, \exp(-50 ((x_1-0.15)^2 + (x_2-0.85)^2 + (x_3-0.7)^2))),
\end{equation*}
with noise standard deviation $\sigma = 0.01$. The initial condition and the forward solutions at two times are shown in \cref{fig:forward3D}.

\begin{figure}[t]
\centering
\begin{tikzpicture}
\node at (-2,3.6) {$t=0$};
\node at (2.7,3.6) {$t=2.4$};
\node at (7.4,3.6) {$t=4$};
\node at (-2,1) {\includegraphics[width=0.27\linewidth]{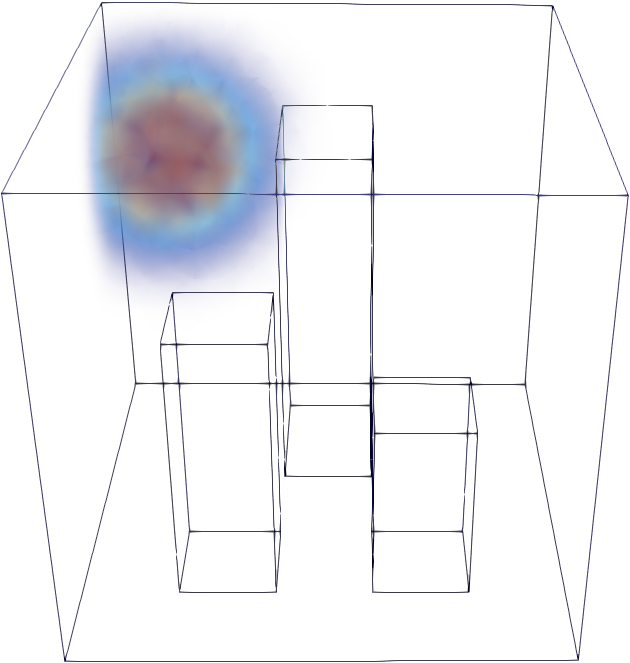}};
\node at (2.7,1) {\includegraphics[width=0.27\linewidth]{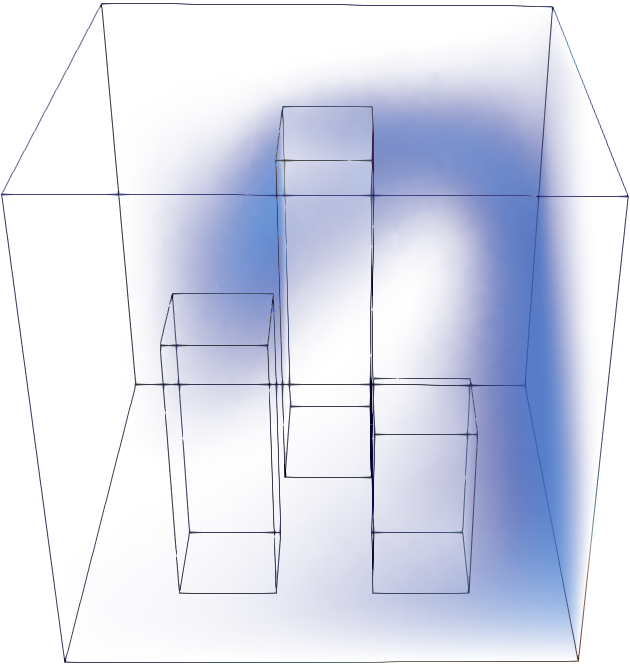}};
\node at (7.4,1) {\includegraphics[width=0.27\linewidth]{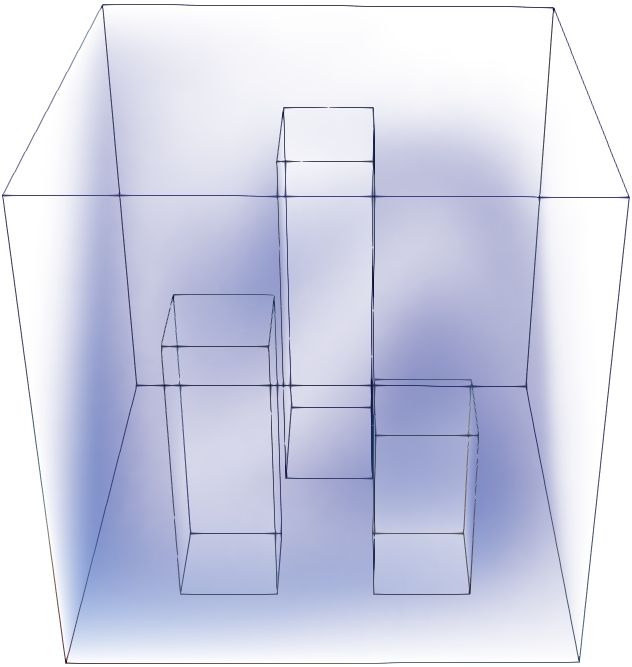}};
\node at (10.3,1) {\includegraphics[width=0.07\linewidth]{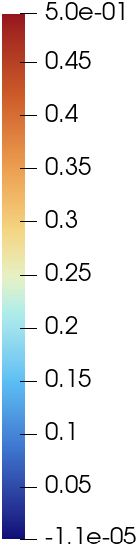}};
\end{tikzpicture}
\label{fig:forward3D}
\caption{The forward solution of the advection diffusion equation in 3D, with the initial condition at the left, the first measurement time $t=2.4$ in the center, and the final time at right.}
\end{figure}
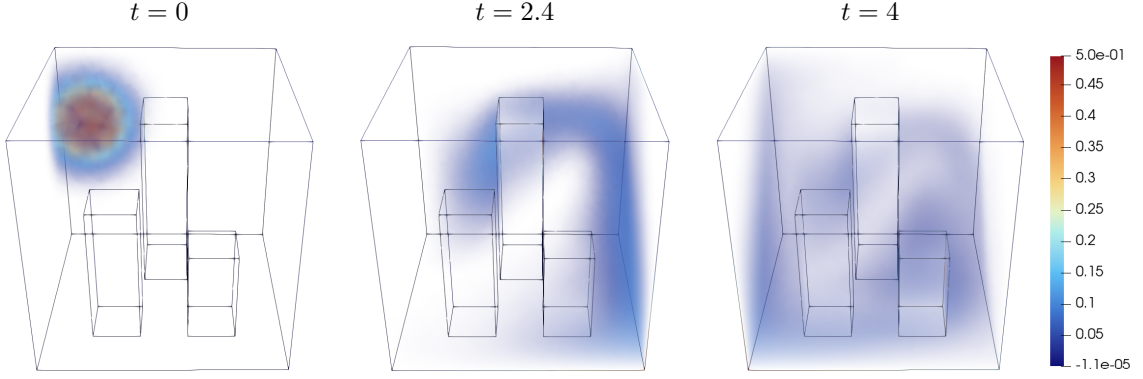

\begin{figure}[t]
\begin{minipage}{0.60\textwidth}
    \centering
    \begin{tikzpicture}
\begin{axis}[
    height=6cm, width=8.7cm,
    xmin=0,
    xmax=250,
    ymin=1e-9,
    ymax=2000,
    compat=1.3,
    xlabel = {$i$},
    ylabel = {$\lambda_i$},
    ymode=log,
    legend pos= north east,
    legend cell align=left,
    legend style={font=\small}
]
    \addplot[color=blue, very thick, mark = none, name path global = unprecon]
	table[x expr={\thisrow{r}},y expr={\thisrow{unprecon}}] {data/spectra_3DAD.txt};
     \addplot[color=orange, very thick, mark = none, name path global = weakest]
	table[x expr={\thisrow{r}},y expr={\thisrow{weakest}}] {data/spectra_3DAD.txt};
    \addplot[color=C1, very thick, mark = none, name path global = prior1]
	table[x expr={\thisrow{r}},y expr={\thisrow{prior1}}] {data/spectra_3DAD.txt};
    \addplot[color=C2, very thick, mark = none, name path global = prior2]
	table[x expr={\thisrow{r}},y expr={\thisrow{prior2}}] {data/spectra_3DAD.txt};
    \addplot[color=C3, very thick, mark = none, name path global = prior3]
	table[x expr={\thisrow{r}},y expr={\thisrow{prior3}}] {data/spectra_3DAD.txt};
    \draw [dashed, C1, name path = cutoff_1] (axis cs:1,3.6e-3) -- (axis cs:250,3.6e-3);
    \draw [dashed, C2, name path = cutoff_2] (axis cs:1,9e-4) -- (axis cs:250,9e-4);
    \draw [dashed, C3, name path = cutoff_3] (axis cs:1,2.25e-4) -- (axis cs:250,2.25e-4);
    \draw [dashed, black, name path = cutoff_min] (axis cs:1,9e-8) -- (axis cs:250,9e-8);
    \path [name intersections={of=unprecon and cutoff_min, by=unpre_min}];
    \path [name intersections={of=unprecon and cutoff_1, by=unpre_1}];
    \path [name intersections={of=unprecon and cutoff_2, by=unpre_2}];
    \path [name intersections={of=unprecon and cutoff_3, by=unpre_3}];
    \draw[dashed] (unpre_min) -- (unpre_min |- current axis.south);
    \draw[dashed, C1] (unpre_1) -- (unpre_1 |- current axis.south);
    \draw[dashed, C2] (unpre_2) -- (unpre_2 |- current axis.south);
    \draw[dashed, C3] (unpre_3) -- (unpre_3 |- current axis.south);
    \path [name intersections={of=weakest and cutoff_min, by=weak_min}];
    \path [name intersections={of=weakest and cutoff_1, by=weak_1}];
    \path [name intersections={of=weakest and cutoff_2, by=weak_2}];
    \path [name intersections={of=weakest and cutoff_3, by=weak_3}];
    \draw[dashed] (weak_min) -- (weak_min |- current axis.south);
    \draw[dashed, C1] (weak_1) -- (weak_1 |- current axis.south);
    \draw[dashed, C2] (weak_2) -- (weak_2 |- current axis.south);
    \draw[dashed, C3] (weak_3) -- (weak_3 |- current axis.south);
    \path [name intersections={of=prior1 and cutoff_1, by=pr1_1}];
    \path [name intersections={of=prior2 and cutoff_2, by=pr2_2}];
    \path [name intersections={of=prior3 and cutoff_3, by=pr3_3}];
    \draw[dashed, C1] (pr1_1) -- (pr1_1 |- current axis.south);
    \draw[dashed, C2] (pr2_2) -- (pr2_2 |- current axis.south);
    \draw[dashed, C3] (pr3_3) -- (pr3_3 |- current axis.south);
    \legend{UP, WP, {PP, $\bs \theta_1$}, {PP, $\bs \theta_2$}, {PP, $\bs \theta_3$}};
    \end{axis}
  \end{tikzpicture}
  \end{minipage}\hfill
\begin{minipage}{0.39\textwidth}
    \centering
    \includegraphics[width=\textwidth] {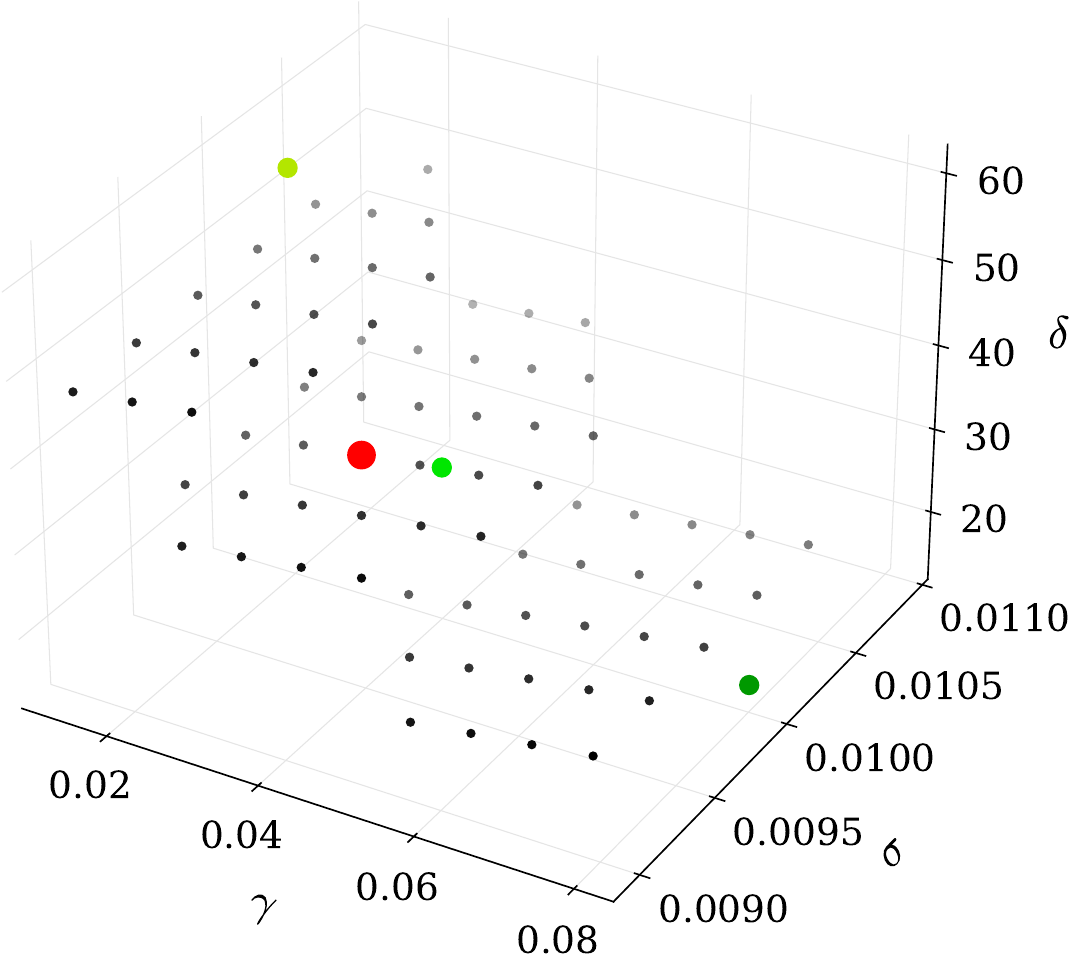}
  \end{minipage}
  \label{fig:spectra_quad_3D}
  \caption{Left: Comparison of the precision update spectra in the 3D advection-diffusion problem, for  UP (blue), WP (orange), and PP with representative values $\bs \theta_1 = [0.02, 60, 0.01]^\top$, $\bs \theta_2 = [0.04, 30, 0.01]^\top$, and $\bs \theta_3 = [0.08, 15, 0.01]^\top$ (light, medium, and dark green). The dashed lines indicate cutoffs defined as in \cref{fig:spectra_error}.
  Right: Quadrature points for $\pi(\bs \theta|\bs y)$, computed using steps of 0.8 standard deviations and a probability threshold of three standard deviations. $\bs \theta_1, \bs \theta_2$, and $\bs \theta_3$ are shown in light, medium, and dark green, and $\bs \theta^\ast = [0.0345, 33.95, 0.00975]^\top$ is shown in red.}
\end{figure}
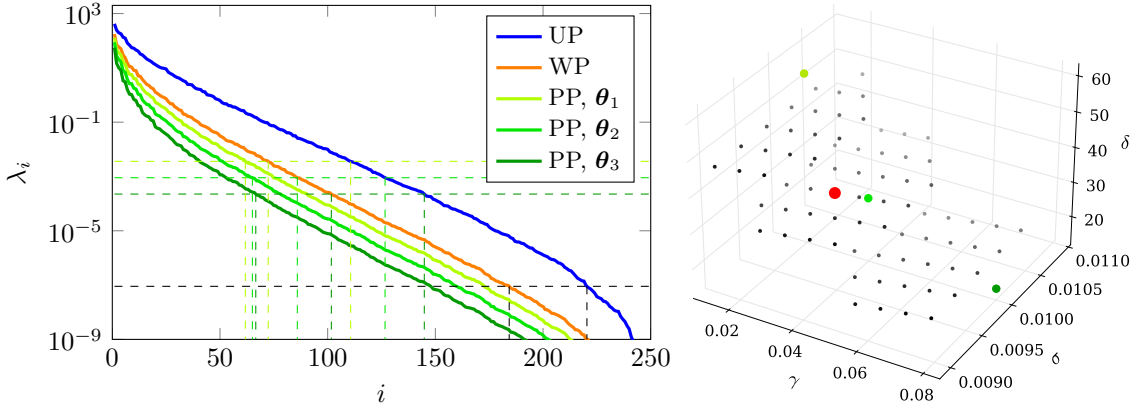

The spectra of the precision updates in PP, WP, and UP are pictured on the left in \cref{fig:spectra_quad_3D}. The hyperpriors are the same as in the 2D problem, except for $\gamma_\text{min}$ increased to 0.01. Note that the decay of the eigenvalues is significantly slower than in the 2D problem, with ranks (listed in \cref{tab:timing3D}) nearly twice as large as those in \cref{tab:timing}. The representative values of $\bs \theta$ for the spectra are chosen from around the high probability region of $\pi(\bs \theta|\bs y)$, approximated by the quadrature points shown at right in \cref{fig:spectra_quad_3D}. 
%
Unlike \cref{fig:quadrature} (left), \cref{fig:spectra_quad_3D} (right) does not show contours of $\pi(\bs \theta|\bs y)$, 
since evaluating the distribution on a grid becomes infeasible for the more expensive 3D problem. 

\cref{tab:timing3D} displays the ranks from \cref{fig:spectra_quad_3D} and the time to compute $N=100$ evaluations of $\pi(\bs \theta|\bs y)$. Like in the 2D problem, the approximations are computed to the minimum cutoff ranks for WP and UP and to a lower rank for PP (here 80), then truncated to the $\bs \theta$-dependent rank for each evaluation. Due to the larger $\gamma_\text{min}$ compared to that in the 2D example, the difference in $\bs \theta$-dependent rank is greater than in 2D: 101 vs 144, compared to 56 vs 69. This leads to the larger difference in timing between WP and UP than in the 2D example. The difference is less pronounced when no CG is used, since then the time is almost entirely spent on the initial precomputation, which uses the more similar $\bs \theta$-independent ranks of 184 and 220.
Both methods are several times faster relative to PP than in the 2D case. 
The cost of forward and adjoint solves increases faster than that of prior solves, applications, and orthogonalization, since the advection-diffusion equation is solved using a quadratic finite element discretization. For example, while 53\% of the time for 100 WP evaluations with CG in 2D is spent performing forward and adjoint solves, the proportion is 89\% for the same computation in 3D.
Thus for more realistic, computationally intensive problems, the advantage of WP and UP over PP grows.

\begin{table}[t]
    \caption{Timing and rank comparison for PP, WP, UP in the 3D advection-diffusion problem. The first row shows time in seconds for $N=100$ evaluations of $\pi(\bs \theta|\bs y)$, and the second row shows the same time if no CG is used in the computation of $\mu_\textup{post}$. The third row shows the minimum $i$ such that $\lambda_i < 10^{-2} \sigma_\textup{min}^2 \delta_\textup{min}^2$. No value is displayed for PP, since $\bs \theta$ is known and the minimum cutoff is not used. The fourth row repeats this for $\lambda_i < 10^{-2} \sigma^2 \delta^2$, using $\bs \theta_3$ as a representative value. The asterisk indicates that the timing had to be extrapolated from $N=5$ (since PP is an entirely online method, every evaluation has a similar cost).}
    \centering
    \begin{tabular}{c|c|c|c}
         & PP & WP & UP \\
         \hline
         time for $N=100$ & $65.63^\ast$ hrs & 2.02 hrs & 3.00 hrs \\
         \hline 
         time for $N=100$, no CG & $56.38^\ast$ hrs & 1.26 hrs & 1.58 hrs \\
        \hline \hline
        rank ($\sigma^2_\text{min}\delta^2_\text{min}$) & --- & 184 & 220 \\
        \hline
        rank ($\sigma^2\delta^2$) & 66 & 101 & 144 \\
        \hline
    \end{tabular}
    \label{tab:timing3D}
\end{table}

\section{Conclusions and discussion} \label{sec:conclusion}
Hierarchical Bayesian inverse problems sit between strictly
linear problems and computationally prohibitive nonlinear inverse problems that require MCMC. Even when the parameter-to-observable map is
linear, rendering the problem conditionally Gaussian, repeated
evaluation of the marginal posteriors poses computational challenges,
and prior work has focused on linear hyperparameters to simplify
precomputation. Our proposed WP and UP methods extend
amortization-by-approximation methods to a broader class of common
hyperparameters, such as correlation lengths, that enter into the
prior nonlinearly. We demonstrate that sacrificing basis optimality in
the low-rank approximation to achieve a $\bs \theta$-independent
approximation is a highly favorable trade-off when the forward
problem is expensive to solve and many evaluations of $\pi(\bs \theta|\bs y)$ are
required.

A natural extension of this work is to accelerate optimization over
$\bs \theta$ by deriving WP and UP approximations of the gradient of
$\pi(\bs \theta|\bs y)$, alongside rigorous error bounds for these
gradient approximations. Additionally, further theoretical analysis of
the truncation error $e_2$ could also clarify which specific problem
classes can bypass the additional cost of CG solves.

A broader extension would address problems where the
parameter-to-observable map $\bs A$ inherently depends on
hyperparameters. Currently, our assumption of a $\bs
\theta$-independent $\bs A$ precludes making the diffusion coefficient or advection velocity in
\cref{sec:adv_diff} hyperparameters. Accommodating $\bs A(\bs \theta)$ would require a
means of efficiently updating the low-rank approximation under
parameter perturbations, perhaps by integrating our framework with
modern surrogate modeling techniques.

Finally, for inverse problems with moderately nonlinear $\bs A(\bs
\theta)$, linearization and Laplace approximations might be applied
following a model similar to INLA, bridging the gap between our
amortized generalized preconditioning approach and existing techniques
for non-Gaussian models. Rigorous analysis quantifying the propagation
of linearization errors through the low-rank updates will be a
critical subject for further investigation.

\section*{Acknowledgments} The authors would like to thank Lisa Gaedke-Merzh\"auser for insightful discussions about INLA, and Gonzalo G.\ de Diego for helpful discussions on the implementation of the numerical examples.

\bibliographystyle{siamplain}
\bibliography{refs}

\end{document}